\documentclass[11pt]{amsart}
\usepackage{fullpage}
\usepackage{amsfonts}
\usepackage{mathrsfs}
\usepackage{mathtools}
\usepackage{leftindex,tensor}
\usepackage{amsmath}
\usepackage{amssymb}
\usepackage{bbm}
\usepackage{latexsym}
\usepackage{lscape}
\usepackage[all,cmtip]{xy}
\usepackage{comment}
\usepackage{appendix}
\usepackage{dsfont}
\usepackage{textcomp}

\usepackage[pagebackref,hyperindex,linktocpage=true]{hyperref}
\usepackage[alphabetic]{amsrefs}

\usepackage{color}

\usepackage[colorinlistoftodos]{todonotes}

\newtheorem{thm}{Theorem}[section]

\newtheorem{prop}[thm]{Proposition}
\newtheorem{lem}[thm]{Lemma}

\newtheorem{lem-def}[thm]{Lemma-Definition}
\newtheorem{cor}[thm]{Corollary}

\theoremstyle{definition}

\newtheorem{rem}[thm]{Remark}
\newtheorem{defn}[thm]{Definition}

\numberwithin{equation}{section}

\newcommand{\into}{\hookrightarrow}

\newcommand{\bbA}{\mathbb{A}}
\newcommand{\bbB}{\mathbb{B}}

\newcommand{\bbG}{\mathbb{G}}

\newcommand{\bbN}{\mathbb{N}}

\newcommand{\bbQ}{\mathbb{Q}}
\newcommand{\bbR}{\mathbb{R}}

\newcommand{\bbT}{\mathbb{T}}

\newcommand{\bbV}{\mathbb{V}}
\newcommand{\bbW}{\mathbb{W}}

\newcommand{\bbZ}{\mathbb{Z}}

\newcommand{\bbg}{\mathbbm{g}}

\newcommand{\bbt}{\mathbbm{t}}

\newcommand{\bfI}{\mathbf{I}}

\newcommand{\bfP}{\mathbf{P}}

\newcommand{\cB}{\mathcal{B}}

\newcommand{\cE}{\mathcal{E}}
\newcommand{\cF}{\mathcal{F}}
\newcommand{\cG}{\mathcal{G}}

\newcommand{\cO}{\mathcal{O}}

\newcommand{\cT}{\mathcal{T}}

\newcommand{\cV}{\mathcal{V}}

\newcommand{\fb}{\mathfrak{b}}

\newcommand{\fg}{\mathfrak{g}}

\newcommand{\fn}{\mathfrak{n}}

\newcommand{\ft}{\mathfrak{t}}

\newcommand{\sA}{\mathsf{A}}
\newcommand{\sB}{\mathsf{B}}

\newcommand{\sG}{\mathsf{G}}

\newcommand{\sP}{\mathsf{P}}

\newcommand{\sT}{\mathsf{T}}

\newcommand{\sV}{\mathsf{V}}

\newcommand{\sX}{\mathsf{X}}

\newcommand{\ga}{\gamma}

\newcommand{\la}{\lambda}
\newcommand{\La}{\Lambda}
\newcommand{\perf}{\mathrm{perf}}
\newcommand{\red}{\mathrm{red}}
\newcommand{\reg}{\mathrm{reg}}
\newcommand{\rs}{\mathrm{rs}}

\DeclareMathOperator{\ad}{ad}
\DeclareMathOperator{\Ad}{Ad}

\DeclareMathOperator{\Gal}{Gal}

\DeclareMathOperator{\Image}{Im}
\DeclareMathOperator{\Lie}{Lie}
\DeclareMathOperator{\SL}{SL}
\DeclareMathOperator{\SO}{SO}
\DeclareMathOperator{\Spec}{Spec}

\DeclareMathOperator{\val}{val}

\title{Equivalued affine springer fibers in mixed characteristic}
\author{Jingren Chi}
\address{Morningside Center of Mathematics and State Key Laboratory of Mathematical Sciences, Academy of Mathematics and Systems Science, Chinese Academy of Sciences, Beijing 100190, China}
\email{jrenchi@amss.ac.cn}
\date{}

\begin{document}
\begin{abstract}
    We study Witt-vector affine Springer fibers for tame equi-valued conjugacy classes in tamely ramified groups. Following the approach of Goresky-Kottwitz-MacPherson in the equal characteristic setting, we show that they admit pavings by iterations of perfections of affine space bundles whose base are smooth Hessenberg varieties. Along the way we revisit a version of the Chevalley restriction theorem for the dual of Lie algebras.
\end{abstract}
\maketitle

\section{Introduction}
Affine Springer fibers are fundamental geometric objects associated to conjugacy classes in the loop groups. Since the initial work of Kazhdan-Lusztig \cite{KL}, they have been intensively studied by many authors from various perspectives. Their geometric properties turn out to be very complicated in general and a good understanding is still lacking. However, the affine Springer fibers for \emph{equi-valued} conjugacy classes are much better understood by the work of Goresky, Kottwitz and MacPherson \cite{GKM}. They showed that these special affine Springer fibers admit pavings by iterated affine space bundles over smooth Hessenberg varieties, and deduced that they have pure cohomology.\par   
Our objective in this paper is to prove similar results in the mixed-characteristic setting. In other words, instead of loop groups we deal with reductive groups over $p$-adic fields and the associated Witt-vector affine Springer fibers that are studied in our previous work \cite{Chi-Witt}. Following \cite{GKM}, we prove our main result Theorem \ref{thm:main} in a more general setting of an arbitrary representation of the group and specialize to the case of the adjoint representation on the Lie algebra at the end (see Corollary \ref{cor:ASF}). Also, we prove the result for general tamely ramified reductive groups (whereas in \cite{GKM} the group is assumed to be split). To check that the equi-valued elements in the Lie algebras satisfy a technical condition in the main Theorem \ref{thm:main}, we need a version of the Chevalley restriction theorem for the dual of a Lie algebra, which is very briefly sketched in \cite{GKM}. If one is willing to impose stronger hypotheses by assuming the existence of a non-degenerate invariant bilinear form on the Lie algebra, it will follow directly from the classical Chevalley restriction theorem. As far as we know, a more general form of the dual Chevalley theorem was first proved by Kac and Weisfeiler \cite{KW}. Recently, Spice and Tsai \cite{ST} clarified the result of Kac and Weisfeiler, with some corrections, and obtained the most general form of the dual Chevalley theorem. We will give an independent and self-contained proof of this result in \S\ref{sec:dual-chevalley}, see Theorem \ref{thm:dual-Chevalley-restriction}. In \S\ref{sec:BT-MP} we first set up the notation and assumptions on reductive groups over local fields and then review some standard facts from Bruhat-Tits theory and Moy-Prasad theory. After introducing the generalizations of Moy-Prasad filtrations in \S\ref{sec:generalization-MP-filtration}, we prove our main result in \S\ref{sec:GASF} and deduce the consequence Corollary \ref{cor:ASF} on Witt-vector affine Springer fibers for tame equi-valued conjugacy classes. 

\subsection*{Acknowledgment}
We heartily thank an anonymous referee for many helpful corrections and suggestions. The first version of this paper is done without using AI tools. During the subsequent revision process, ChatGPT 6 Astra has been used for proofreading and text improvements. This work was partially supported by National Key R\&D Program of China No.2023YFA1009701, 
National Natural Science Foundation of China (No.12288201, No.12231001), CAS Project for Young Scientists in Basic Research (Grant No.YSBR-033).

\section{The dual Chevalley restriction theorem and application}\label{sec:dual-chevalley}
Our goal in this section is to revisit the Chevalley restriction theorem for the dual of a reductive Lie algebra and deduce its consequence on good vectors in the Lie algebra. We give a self-contained proof by following the approaches to the classical Chevalley theorem in \cite{Jan} and \cite{BC22}. We refer the reader to \cite{KW} and \cite{ST} for closely related approaches.

\subsection{The setup}\label{subsec:setup}
In this section we let $G$ be a connected reductive group over an algebraically closed field $k$. Let $T\subset G$ be a maximal torus and $B\subset G$ a Borel subgroup containing $T$ with unipotent radical $U\coloneqq[B,B]$. Denote their Lie algebras by $\fg\coloneqq\Lie(G)$, $\ft\coloneqq\Lie(T)$, $\fb\coloneqq\Lie(B)$, $\fn\coloneqq\Lie(U)$. Let $W\coloneqq N_G(T)/T$ be the Weyl group. For each $w\in W$, choose a representative $\dot{w}\in N_G(T)\subset G$. Let $X^*(T)$ be the character lattice and $X_*(T)$ be the cocharacter lattice of $T$. Let $\Ad$ (resp. $\Ad^*$) denote the adjoint (resp. co-adjoint) action of $G$ on $\fg$ (resp. the dual $\fg^*$) and let $\ad$ (resp. $\ad^*$) denote the induced action of $\fg$ on $\fg$ (resp. $\fg^*$).\par  
Let $\Phi\subset X^*(T)$ be the set of roots of $T$ in $\fg$. We have a decomposition $\Phi=\Phi^+\sqcup\Phi^-$ where $\Phi^+$ is the set of positive roots (those appearing in $\fb$) and $\Phi^-=-\Phi^+$ is the set of negative roots. Let $\Delta\subset\Phi^+$ be the set of simple roots. For each $\alpha\in\Phi$, let $\alpha^\vee\in X_*(T)$ be the corresponding coroot and let $d\alpha^\vee\colon k\to\ft$ be its differential. The \emph{height} of a root $\alpha\in\Phi$ is the integer $\mathrm{ht}(\alpha)$ defined as the sum of coefficients in the expression of $\alpha$ as a linear combination of simple roots. \par 
We have the root space decomposition \[\fg=\ft\oplus\bigoplus_{\alpha\in\Phi}\fg_\alpha\] 
where $\fg_\alpha=\Lie(U_\alpha)$, with $U_\alpha\subset G$ the corresponding root subgroup. 
For any $\alpha\in\Phi^+$, we can choose a homomorphism $i_\alpha\colon\SL_2\to G$ whose restriction to the upper (resp. lower) triangular unipotent matrices gives isomorphisms $x_\alpha\colon\bbG_a\to U_\alpha$ (resp. $x_{-\alpha}\colon\bbG_a\to U_{-\alpha}$) onto the respective root groups, and whose restriction to the diagonal $\bbG_m$ is equal to $\alpha^\vee$. Then we get elements $X_\alpha\coloneqq dx_\alpha(1)\in\fg_\alpha\setminus\{0\}$ for all $\alpha\in\Phi$. By the standard commutation relations (see for example \cite{SGA3IIInew}*{Expos\'e XXII, Lemme 5.4.9}), for any \emph{linearly independent} roots $\alpha,\beta\in\Phi$ we have
\begin{equation}\label{eq:commutation-alpha-beta}
    \Ad(x_\alpha(t))X_\beta-X_\beta\in\bigoplus_{\substack{i\in\bbZ_{>0}\\ i\alpha+\beta\in\Phi}}\fg_{i\alpha+\beta},\quad\forall t\in k.
\end{equation}
On the other hand, for any $\alpha\in\Phi^+$, using the homomorphism $i_\alpha$ to reduce the calculation to the $\SL_2$ case, we obtain
\begin{equation}\label{eq:commutation-alpha-minus-alpha}
    \Ad(x_\alpha(t))X_{-\alpha}=t\cdot d\alpha^\vee(1)+X_{-\alpha}-t^2 X_\alpha,\quad\forall t\in k.
\end{equation}
\subsection{The usual Chevalley restriction theorem}
We say that the group $G$ is \emph{root-smooth} if the morphism $\alpha\colon T\to\bbG_m$ is smooth for any root $\alpha\in\Phi$, see \cite{BC22}*{4.1.1}. This is equivalent to requiring that $d\alpha\ne0$ for each root $\alpha\in\Phi$. This condition is satisfied if $\mathrm{char}(k)\ne2$ or if $G$ has no $\mathrm{Sp}_{2n}$ factor, and it ensures that the set $\fg_{\rs}$ of regular semisimple elements in $\fg$ is nonempty. 
\begin{thm}[Chevalley Restriction]\label{thm:chevalley}
    Restriction along the natural embedding $\ft\into\fg$ induces an injective $k$-algebra homomorphism
    \[k[\fg]^G\into k[\ft]^W\]
    which is bijective if $G$ is root-smooth.
\end{thm}
See \cite{BC22}*{Theorem 4.1.10} for a much more general statement, and \emph{loc.cit.} Remark 4.1.11 for the failure of surjectivity without the root-smoothness assumption.

\subsection{Regular semisimple elements in the dual}
Since the Lie algebra $\ft$ is a direct summand of $\fg$, its dual $\ft^*$ is a direct summand of $\fg^*$. We say that an element of $\fg^*$ is \emph{semisimple} if it is in the co-adjoint $G$-orbit of some element in $\ft^*$. An element in $\ft^*$ is \emph{regular} if the identity component of its stabilizer in $G$ (under the co-adjoint action) is equal to $T$. Let $\ft^*_{\reg}\subset\ft^*$ be the subset of regular elements. An element in $\fg^*$ is \emph{regular semisimple} if it lies in the $G$-orbit of some element in $\ft^*_{\reg}$. We let $\fg^*_{\rs}\subset\fg^*$ denote the subset of regular semisimple elements.\par 
For any element $\xi\in\fg^*$, we let $G_\xi$ be its stabilizer in $G$ (for the co-adjoint action) and let 
\[\fg_\xi\coloneqq\{x\in\fg\mid\xi([x,y])=0,\ \forall y\in\fg\}\]
be its infinitesimal stabilizer. 
\begin{lem}
    Let $\xi\in\ft^*$ and let $\Phi_\xi\coloneqq\{\alpha\in\Phi\mid \xi(d\alpha^\vee(1))=0\}$. 
    \begin{enumerate}
        \item We have $\fg_\xi=\ft\oplus\bigoplus\limits_{\alpha\in\Phi_\xi}\fg_\alpha$.
        \item The stabilizer $G_\xi$ is a smooth affine algebraic group over $k$ and we have $\Lie(G_\xi)=\fg_\xi$. 
    \end{enumerate}
\end{lem}
\begin{proof}
    (1) It is clear that $\ft\subset\fg_\xi$ and $\fg_\xi$ is a $T$-stable subspace of $\fg$. Hence $\fg_\xi$ is a direct sum of $\ft$ and certain root spaces. Then the desired decomposition follows easily. \par
    (2) Let $G_\xi^{\red}$ be the reduced induced subscheme of $G_\xi$. Then it is clear that $T\subset G_\xi^{\red}$. Thanks to \eqref{eq:commutation-alpha-beta}, \eqref{eq:commutation-alpha-minus-alpha} we also have $U_\alpha\subset G_\xi^{\red}$ for all $\alpha\in\Phi_\xi$. Combined with (1) we get that
    \[\dim(\Lie(G_\xi))\ge\dim (G_\xi^{\red})\ge\dim\fg_\xi.\]
    Since there is an obvious inclusion $\Lie(G_\xi)\subset\fg_\xi$ we conclude that $\Lie(G_\xi)=\fg_\xi$ and $G_\xi=G_\xi^{\red}$. So $G_\xi$ is smooth.  
\end{proof}
As an immediate consequence we get:
\begin{cor}\label{cor:reg-t-dual-criterion}
    An element $\eta\in\ft^*$ is regular if and only if $\eta(d\alpha^\vee(1))\ne0$ for each root $\alpha\in\Phi$, in which case we have $G_\eta^\circ=T$ (the identity component of the stabilizer $G_\eta$). In particular, $\ft^*_{\reg}$ is a $W$-invariant open subset of $\ft^*$ and it is nonempty if and only if $d\alpha^\vee\ne0$ for all $\alpha\in\Phi$. 
\end{cor}
This leads us to the following notion dual to ``root-smoothness".
\begin{defn}
    We say that the reductive group $G$ is \emph{coroot-unramified} if for each root $\alpha\in\Phi$, the corresponding coroot $\alpha^\vee\colon\bbG_m\to T$ is an unramified morphism, or equivalently $d\alpha^\vee\ne0$ for all $\alpha\in\Phi$.
\end{defn}
\begin{rem}
    The group $G$ is coroot-unramified if either $\mathrm{char}(k)\ne2$ or $G$ has no $\SO(2n+1)$ factor. 
\end{rem}

\subsection{The dual Grothendieck-Springer alteration}
Let 
\[\fn^\perp\coloneqq(\fg/\fn)^*\subset\fg^*\] 
be the orthogonal complement of $\fn$ in $\fg^*$. The \emph{Grothendieck-Springer alteration for $\fg^*$} is defined as functor of points by
\[\widetilde{\fg}^*\coloneqq\{(gB,\xi)\in G/B\times\fg^*\mid\Ad^*(g)^{-1}\xi\in\fn^\perp\}\xrightarrow{\pi_{\fg^*}}\fg^*\]
where the map $\pi_{\fg^*}$ sends a pair $(gB,\xi)\in\widetilde{\fg}^*$ to $\xi\in\fg^*$. The fiber of $\pi_{\fg^*}$ over a point $\xi\in\fg^*$ is described by
\[\pi_{\fg^*}^{-1}(\xi)=\{gB\in G/B\mid \Ad^*(g)^{-1}\xi\in\fn^\perp\}.\]

There is a natural $G$-action on $\widetilde{\fg}^*$ such that $\pi_{\fg^*}$ is $G$-equivariant: $h\in G$ acts by $(gB,\xi)\mapsto(hgB,\Ad^*(h)\xi)$. By definition $\widetilde{\fg}^*$ is a closed subscheme of $G/B\times\fg^*$ and since the flag variety $G/B$ is proper, the map $\pi_{\fg^*}$ is proper. We also have a canonical isomorphism
\[\widetilde{\fg}^*\cong G\times^B\fn^\perp\coloneqq (G\times\fn^\perp)/B,\quad (gB,\xi)\mapsto (g,\Ad^*(g)^{-1}\xi)\]
where $B$ acts anti-diagonally. Thus $\widetilde{\fg}^*$ is a vector bundle over $G/B$ whose fibers are isomorphic to $\fn^\perp$ and we have
\[\dim(\widetilde{\fg}^*)=\dim(G/B)+\dim(\fn^\perp)=\dim(\fg^*).\]
\begin{lem}\label{lem:GS-alteration-surjective}
    The map $\pi_{\fg^*}$ is surjective. 
\end{lem}
\begin{proof}
    Extend the collection of root vectors $\{X_\alpha,\alpha\in\Phi\}$ in \S\ref{subsec:setup} to a basis of $\fg=\ft\oplus\bigoplus\limits_{\alpha\in\Phi}\fg_\alpha$. For each $\alpha\in\Phi$, let $\xi_\alpha\in\fg^*$ be the vector in the dual basis corresponding to $X_\alpha$. Consider the element $\xi_-\coloneqq\sum\limits_{\alpha\in\Delta}\xi_{-\alpha}$ where the summation is over the set of simple roots. Then we have $\Ad^*(b)\xi_-\in\fn^\perp$ for all $b\in B$.\par 
    On the other hand, let $g\in G$ be such that $\Ad^*(g)\xi_-\in\fn^\perp$, so we have $\xi_-(\Ad(g)^{-1}(X))=0$ for all $X\in\fn$. Write $g^{-1}=u\dot{w}b$ under the Bruhat decomposition, where $u\in U,b\in B$ and $w\in W$. Suppose that $w\ne1$ and let $\alpha\in\Delta$ be a simple root such that $w^{-1}(\alpha)\in\Phi^-$. Let $X\coloneqq\Ad(\dot{w}b)^{-1}X_{-\alpha}$. Then we have $X\in\fn$ and $\Ad(g)^{-1}X=\Ad(u)X_{-\alpha}$. So $\xi_-(\Ad(u)X_{-\alpha})=0$. But using the equations \eqref{eq:commutation-alpha-beta}, \eqref{eq:commutation-alpha-minus-alpha} and writing $u$ as product of elements in the root groups, we deduce that 
    $\xi_-(\Ad(u)X_{-\alpha})=1$ which is a contradiction. Thus we have $w=1$ and therefore $g\in B$. \par 
    We conclude that the fiber $\pi_{\fg^*}^{-1}(\xi_-)$ consists of a single point. Since $\pi_{\fg^*}$ is proper by upper semi-continuity its fiber has dimension $0$ over a dense open subset of its image. Therefore we have $\dim(\pi_{\fg^*}(\widetilde{\fg}^*))=\dim(\widetilde{\fg}^*)=\dim(\fg^*)$. This implies that $\pi_{\fg^*}$ is dominant, and being proper it must be surjective.
\end{proof}

\begin{lem}\label{lem:rs-Springer-fiber}
    Suppose that $G$ is coroot-unramified. For any $\xi\in\ft^*_{\reg}$ we have 
    \[\pi_{\fg^*}^{-1}(\xi)=\{\dot{w}B\mid w\in W\}.\]
\end{lem}
\begin{proof}
    The containment ``$\supset$" follows by definition and it remains to show the reverse inclusion. Take any $gB\in\pi_{\fg^*}^{-1}(\xi)$ and we need to show that $gB=\dot{w}B$ for some $w\in W$. Let $B^-\subset G$ be the Borel subgroup opposite to $B$ and let $U^-\subset B^-$ be its unipotent radical with Lie algebra $\fn^-\coloneqq\Lie(U^-)$. By the Bruhat decomposition we may assume that $g=u\dot{w}$ for some $w\in W$ and $u\in U$ such that $\dot{w}^{-1}u\dot{w}\in U^-$. Consider the set
    \[\Phi_w^+\coloneqq\{\alpha\in\Phi^+\mid w^{-1}(\alpha)\in\Phi^-\}=\{\alpha_1,\dotsc,\alpha_s\}\]
    where $s=\ell(w)$ and the order is chosen so that the height is non-decreasing, i.e. $\mathrm{ht}(\alpha_i)\le\mathrm{ht}(\alpha_j)$ for all $1\le i<j\le s$.
    Then we can write 
    \[u=\prod_{j=1}^s x_{\alpha_j}(u_j),\qquad u_j\in k\;\;\forall j.\]
    For any $1\le i\le s$, we have 
    \[\Ad(\dot{w})(X_{-w^{-1}(\alpha_i)})=c_iX_{-\alpha_i},\qquad c_i\in k^\times.\] 
    Then from the identities \eqref{eq:commutation-alpha-beta} and \eqref{eq:commutation-alpha-minus-alpha} we deduce that
    \begin{align*}
        \Ad(g)(X_{-w^{-1}(\alpha_i)})&=c_i\cdot\Ad(u)(X_{-\alpha_i})\\
        &\in c_iu_i\cdot d\alpha_i^\vee(1)+\sum_{j<i}b_{ij}\cdot d\alpha_j^\vee(1)+\fn\oplus\fn^-
    \end{align*}
    where $b_{ij}$ are polynomials in the variables $u_1,\dotsc, u_{i-1}$ with vanishing constant term. By assumption we have $\Ad^*(g)^{-1}(\xi)\in\fn^\perp$ and so 
    \[\xi(\Ad(g)(X_{-w^{-1}(\alpha_i)}))=0\qquad \forall1\le i\le s.\]
    Since $\xi\in\ft^*$, it vanishes on any element in $\fn\oplus\fn^-$ and thus the identity above implies by induction that $u_i\cdot\xi(d\alpha_i^\vee(1))=0$. On the other hand, since $\xi$ is regular we have $\xi(d\alpha_i^\vee(1))\ne0$ and therefore $u_i=0$ for all $i$. This shows that $u=1$ and $g=\dot{w}$, so the proof is completed.
\end{proof}

\begin{prop}\label{prop:n-perp-rs-isom}
    Suppose that $G$ is coroot-unramified. Let 
    \[\fn^\perp_{\rs}\coloneqq\fn^\perp\cap\fg^*_{\rs}\] 
    and let $q\colon\fn^\perp\to\ft^*$ denote the natural projection which is dual to the embedding $\ft\cong\fb/\fn\into\fg/\fn$. Then we have $\fn^\perp_{\rs}=q^{-1}(\ft^*_{\reg})$.\par 
    Moreover, the coadjoint action induces an isomorphism
    \[U\times\ft^*_{\reg}\xrightarrow{\sim}\fn^\perp_{\rs},\quad (u,\eta)\mapsto\Ad^*(u)\eta.\]
\end{prop}
\begin{proof}
    From the well-known fact that the adjoint action of $B$ preserves the fibers of the projection $\fb\to\fb/\fn\cong\ft$, we deduce that the coadjoint action of $B$ on $\fn^\perp$ preserves the fibers of $q$.\par 
    For any element $\eta\in\fn^\perp_{\rs}$, by definition there exists $g\in G$ such that $\xi\coloneqq\Ad^*(g)\eta\in\ft^*_{\reg}$. Then we have $gB\in\pi_{\fg^*}^{-1}(\xi)$ and by Lemma \ref{lem:rs-Springer-fiber} there exists $w\in W$ and $b\in B$ such that $g=\dot{w}b$. Then using that the fibers of $q$ are preserved by $B$ we deduce that 
    \[\eta=\Ad^*(g)^{-1}(\xi)=\Ad^*(b)^{-1}\Ad^*(\dot{w})^{-1}(\xi)\in q^{-1}(\Ad^*(\dot{w})^{-1}(\xi))\subset q^{-1}(\ft^*_{\reg})\]
    where the last inclusion follows from the $W$-invariance of $\ft^*_{\reg}$ in Corollary \ref{cor:reg-t-dual-criterion}. This proves the inclusion $\fn^\perp_{\rs}\subset q^{-1}(\ft^*_{\reg})$. Composing this inclusion with the coadjoint action map we get a morphism
    \[\varphi\colon U\times\ft^*_{\reg}\to q^{-1}(\ft^*_{\reg}).\]
    Its differential at an element $(1,\xi)$, where $\xi\in\ft^*_{\reg}$, is given by
    \[d\varphi_{(1,\xi)}\colon\fn\oplus\ft^*\to\fn^\perp,\quad(X,\eta)\mapsto\eta+\ad^*(X)\xi\]
    where $(\ad^*(X)\xi)(Y)\coloneqq-\xi([X,Y])$ for all $Y\in\fg$. One verifies easily that $d\varphi_{(1,\xi)}$ is injective and hence bijective since the two sides have the same dimension. Thus $\varphi$ is \'etale. From the definition we see that $\varphi$ commutes with the natural projections to $\ft^*_{\reg}$ on both sides. By Corollary \ref{cor:reg-t-dual-criterion} $U$ intersects trivially with the stabilizers of elements in $\ft^*_{\reg}$, so the map $\varphi$ is injective and hence an open embedding. On the other hand, by \cite{Borel-LAG}*{Chapter I, Proposition 4.10}, the restriction of $\varphi$ to each fiber over $\ft^*_{\reg}$ is a closed embedding and hence must be an isomorphism. Then we deduce that $\varphi$ is an isomorphism by the fiberwise criterion \cite{EGAIV4}*{Corollaire 17.9.5}. This also shows the equality $\fn^\perp_{\rs}=q^{-1}(\ft^*_{\reg})$ which concludes the proof.
\end{proof}

Let $\widetilde{\fg}^*_{\rs}\coloneqq\pi_{\fg^*}^{-1}(\fg^*_\rs)$ and let $\pi_{\fg^*_{\rs}}\colon \widetilde{\fg}^*_{\rs}\to\fg^*_{\rs}$ be the base change of $\pi_{\fg^*}$. We consider the following commutative diagram in which all maps are $G$-equivariant:
\[\xymatrix@C=2cm{
& G/B & \\
G/T\times\ft^*_{\reg}\ar[ur]^{p_1}\ar[rr]^{f\colon (gT,\eta)\mapsto(gB,\Ad^*(g)\eta)}_{\sim}\ar[rd]_{\pi'\colon(gT,\eta)\mapsto\Ad^*(g)\eta} & & \widetilde{\fg}^*_{\rs}\ar[ul]_{p_2}\ar[ld]^{\pi_{\fg^*_{\rs}}}\\
&\fg^*_{\rs}&
}\]
where $G$ acts on $G/T\times\ft^*_{\reg}$ by left multiplication on $G/T$ (and acting trivially on $\ft^*_{\reg}$) and the maps $p_1,p_2$ are the natural projections. \par

\begin{lem}\label{lem:rs-W-torsor}
    Suppose that $G$ is coroot-unramified. In the diagram above, the middle horizontal map $f$ is an isomorphism. Moreover, if we equip  $\widetilde{\fg}^*_{\rs}$ with the $W$-action induced via $f$ from the $W$-action on $G/T\times\ft^*_{\reg}$ given by the rule $(gT,\xi)\mapsto (g\dot{w}^{-1}T,\Ad^*(\dot{w})\xi)$, then the maps $\pi'$ and $\pi_{\fg^*_{\rs}}$ are finite \'etale $W$-torsors. Finally $\fg^*_{\rs}$ is a dense open subset of $\fg^*$. 
\end{lem}
\begin{proof}
     To show that $f$ is an isomorphism, by $G$-equivariance it suffices to show that it is an isomorphism when restricted to the inverse image over the open cell $U^-B/B\subset G/B$ (since $G/B$ is covered by the translates of the open cell). Clearly we have isomorphisms
    \[p_1^{-1}(U^-B/B)=U^-B/T\times\ft^*_{\reg}\cong U^-\times U\times\ft^*_{\reg},\]
    \[p_2^{-1}(U^-B/B)=\{(u,\xi)\in U^-\times\fg^*_{\rs}\mid\Ad^*(u)^{-1}\xi\in\fn^\perp_{\rs}\}\cong U^-\times\fn^\perp_{\rs}\]
    where the map in the second isomorphism sends a pair $(u,\xi)$ to $(u,\Ad^*(u)^{-1}\xi)$. Under these isomorphisms, the restriction of the middle map $f$ becomes the map
    \[U^-\times U\times\ft^*_{\reg}\to U^-\times\fn^\perp_{\rs},\quad (u_1,u_2,\eta)\mapsto(u_1,\Ad^*(u_2)\eta)\]
    which is an isomorphism by Proposition \ref{prop:n-perp-rs-isom}. So $f$ is indeed an isomorphism.\par 
    Let $\pi''\colon G/T\times\ft_\reg^*\to\fg^*$ be the composition of the lower-left map $\pi'$ with the inclusion $\fg_{\rs}^*\into\fg^*$. Then $\pi''$ is also $G$-equivariant and its differential at a point $(1,\xi)$, where $\xi\in\ft^*_{\reg}$, is given by
    \[d\pi''_{(1,\xi)}\colon \fg/\ft\oplus\ft^*\to\fg^*,\quad (X,\eta)\mapsto\ad^*(X)(\xi)+\eta.\]
    One easily shows that it is injective and hence bijective as the two sides have the same dimension. Then by $G$-equivariance we see that $\pi''$ is \'etale. In particular its image $\fg^*_\rs$ is a nonempty open subset of $\fg^*$ and therefore dense. \par 
    Since $f$ is an isomorphism and the diagram commutes, we deduce that $\pi_{\fg^*_{\rs}}$ is \'etale, and being proper, it is hence finite \'etale and the same holds for $\pi'$. \par 
    The rule $(gT,\xi)\mapsto (g\dot{w}^{-1}T,\Ad^*(\dot{w})\xi)$ defines a $W$-action on $G/T\times\ft^*_{\reg}$ that commutes with the $G$-action and such that $\pi'$ is $W$-invariant (with $W$ acting trivially on $\fg^*_{\rs}$). For any $\xi\in\ft^*_{\reg}$, one checks directly by definition that 
    \[(\pi')^{-1}(\xi)=\{(\dot{w}^{-1}T,\Ad^*(\dot{w})\xi)\mid w\in W\}\] 
    (alternatively one can use the isomorphism $f$ and Lemma \ref{lem:rs-Springer-fiber}). So the $W$-action on the fiber $(\pi')^{-1}(\xi)$ is simply transitive. Using $G$-translation and the fact that the $W$-action commutes with the $G$-action, we see that $W$ acts simply transitively on all fibers of $\pi'$. Therefore $\pi'$ is a Galois \'etale $W$-torsor and by transport of structure through $f$, the same holds for $\pi_{\fg^*_{\rs}}$.
\end{proof}

\subsection{The dual Chevalley theorem}
Let $\fg^*_0\subset\fg^*$ be the open subset where $\pi_{\fg^*}$ has $0$-dimensional fibers. Since $\pi_{\fg^*}$ is $G$-equivariant, $\fg^*_0$ is a $G$-stable open subset of $\fg^*$. Suppose that $G$ is coroot-unramified. Then Lemma \ref{lem:rs-Springer-fiber} shows that $\fg^*_{\rs}\subset\fg^*_0$ and Lemma \ref{lem:rs-W-torsor} shows that it is a dense open subset. Let $\widetilde{\fg}^*_0\coloneqq\pi_{\fg^*}^{-1}(\fg^*_0)$ and let $\pi_{\fg^*_0}\colon\widetilde{\fg}^*_0\to\fg^*_0$ be the base change of $\pi_{\fg^*}$. Then we have
\[\dim(\widetilde{\fg}^*\setminus\widetilde{\fg}^*_0)\le\dim(\widetilde{\fg}^*)-1=\dim(\fg^*)-1.\]
On the other hand, since the fibers of $\pi_{\fg^*}$ over the complement $\fg^*\setminus\fg^*_0$ have positive dimension, we deduce that
\begin{equation}\label{eq:dim-g-0-complement}
    \dim(\fg^*\setminus\fg^*_0)\le\dim(\widetilde{\fg}^*\setminus\widetilde{\fg}^*_0)-1\le\dim(\fg^*)-2.
\end{equation}
Since $\pi_{\fg^*_0}$ is proper and has $0$-dimensional fibers, it is finite and so integral. Since $\widetilde{\fg}^*_0$ is smooth and a fortiori normal, it is equal to the normalization of $\fg^*_0$ in $\widetilde{\fg}^*_{\rs}$. Thus the $W$ action on $\widetilde{\fg}^*_{\rs}$ extends uniquely to $\widetilde{\fg}^*_0$ and the map $\pi_{\fg^*}$ induces a finite morphism from the coarse quotient $\widetilde{\fg}^*_0/W$ to $\fg^*_0$ whose restriction along the dense open subset $\fg^*_{\rs}$ is an isomorphism due to Lemma \ref{lem:rs-W-torsor}. Then it is a finite birational morphism and the normality of $\fg^*_0$ implies that it is an isomorphism $\widetilde{\fg}^*_0/W\cong\fg^*_0$, see \cite{SP}*{\href{https://stacks.math.columbia.edu/tag/0AB1}{Lemma 0AB1}}.\par 
We also remark that since $\widetilde{\fg}^*_0$ and $\fg^*_0$ are smooth (over $k$), the miracle flatness \cite{EGAIV2}*{Proposition 6.1.5} or \cite{SP}*{\href{https://stacks.math.columbia.edu/tag/00R4}{Lemma 00R4}} implies that $\pi_{\fg^*_0}$ is a finite flat morphism. 
\begin{thm}[Dual Chevalley Restriction]\label{thm:dual-Chevalley-restriction}
    The embedding $\ft^*\into\fg^*$ induces an injective homomorphism of $k$-algebras 
    \[k[\fg^*]^G\into k[\ft^*]^{W}\]
    which is bijective if $G$ is coroot-unramified.  
\end{thm}
\begin{proof}
    From Lemma \ref{lem:GS-alteration-surjective} we deduce that the co-adjoint action map 
    \[G\times\fn^\perp\to\fg^*\] 
    is surjective. This implies that the natural restriction map $k[\fg^*]^G\to k[\fn^\perp]$ is injective. We claim that its image lies in $k[\ft^*]^W$. Here we are using the surjection $\fn^\perp=(\fg/\fn)^*\to\ft^*$ that is dual to the natural embedding $\ft\cong\fb/\fn\into\fg/\fn$. We need to show that for any function $f\in k[\fg^*]^G$, its restriction to $\fn^\perp$ is invariant under translation by $\fb^\perp=\ker(\fn^\perp\to\ft^*)$. Let $\la\in X_*(T)$ be the sum of all the positive coroots. Then $\langle\la,\alpha\rangle>0$ for each positive root $\alpha\in\Phi^+$. We let $\bbG_m$ act on $\fg^*$ through the composition of $\la$ and the co-adjoint representation $\Ad^*$. Then any element $\xi\in\fn^\perp$ is a sum of vectors of non-negative weights and we decompose $\xi=\xi_0+\xi_1$ where $\xi_0$ has weight $0$ and $\xi_1\in\fb^\perp$ is a sum of vectors of positive weights. By $G$-invariance we get
    \[f(\xi)=\lim_{t\to0}f(\Ad^*(\la(t))\xi)=\lim_{t\to0}f(\xi_0+\Ad^*(\la(t))\xi_1)=f(\xi_0)\]
    here the limit means that the expression, viewed as a function in $t\in\bbG_m$, extends to a polynomial function on $\bbA^1$ and we take the value at $0$ of the extended function. This shows that the restriction of $f$ to $\fn^\perp$ lies in $k[\ft^*]$. By the $G$-invariance, it lies further in the subalgebra of $W$-invariants $k[\ft^*]^W$. Thus the natural map $k[\fg^*]^G\to k[\ft^*]^W$, whose composition with $k[\ft^*]^W\into k[\fn^\perp]$ is injective as we observed above, is itself injective.\par 
    Now assume that $G$ is coroot-unramified and it remains to show the surjectivity. Consider the natural morphism 
    \[\tilde{\pi}\colon\widetilde{\fg}^*\cong G\times^B\fn^\perp\to\ft^*\]
    induced by the map $q\colon\fn^\perp\to\ft^*$ that is dual to the embedding $\ft\cong\fb/\fn\into\fg/\fn$. The restriction of $\tilde{\pi}$ to $\widetilde{\fg}^*_{\rs}$ is $W$-equivariant by the description in Lemma \ref{lem:rs-W-torsor}, and so its restriction to $\widetilde{\fg}^*_0$ is also $W$-equivariant. For any $W$-invariant function $\varphi\in k[\ft^*]^W$, its pullback $\tilde{\varphi}=\tilde{\pi}^*\varphi$ is a $W$-invariant function on $\widetilde{\fg}^*_0$, and due to the isomorphism $\widetilde{\fg}^*_0/W\cong\fg^*_0$ descends to a function $\psi\in k[\fg^*_0]=k[\fg^*]$ where the equality comes from the dimension estimate \eqref{eq:dim-g-0-complement} and Hartog's extension principle. Since $\tilde{\varphi}$ is $G$-invariant by construction, the function $\psi$ is also $G$-invariant. By definition of the map $\tilde{\pi}$, the restriction of the function $\psi\in k[\fg^*]^G$ along the embedding $\ft^*\into\fg^*$ coincides with $\varphi$ on the dense open subset $\ft^*_\reg\subset\ft^*$ and hence equals $\varphi$. This finishes the proof.
\end{proof}

\subsection{Good elements in Lie algebras}\label{subsec:good-elements}
We first recall some basic definitions from \cite{GKM}*{\S2}. Let $\rho\colon G\to\mathrm{GL}(V)$ be a finite dimensional algebraic representation of $G$ and let $V^*$ be the dual representation. For each $\chi\in X^*(T)$, let $V_\chi$ be the corresponding weight space of $V$. 
\begin{defn}\label{def:filtration}
    For any $y\in X_*(T)_\bbR$, we define two decreasing $\bbR$ filtrations on $V$ by 
    \[F_y^tV\coloneqq\bigoplus_{\substack{\chi\in X^*(T),\\ \chi(y)\ge t}}V_\chi,\qquad F_y^{>t}V\coloneqq\bigoplus_{\substack{\chi\in X^*(T),\\ \chi(y)> t}}V_\chi.\]
\end{defn}

\begin{defn}
    A vector $v\in V$ is \emph{$G$-unstable} if there exists $g\in G$, $y\in X_*(T)_\bbR$ such that $v\in \rho(g)F_y^{>0}V$. Equivalently, $v\in V$ is $G$-unstable if there exists a one-parameter subgroup $\la\colon\bbG_m\to G$ such that $\lim\limits_{t\to0}\rho(\la(t))v=0$.
\end{defn}
\begin{defn}\label{def:good-vector}
    A vector $v\in V$ is \emph{$G$-good} if there is no nonzero $G$-unstable vector $v^*$ in the dual representation $V^*$ that vanishes identically on the subspace $\fg\cdot v$ of $V$.
\end{defn}
Since any two maximal tori in $G$ are conjugate, the notions of $G$-unstable and $G$-good do not depend on the choice of $T$.\par 
\begin{lem}\label{lem:G-good-vector}
    A vector $v\in V$ is $G$-good if and only if for any $g\in G$ and $y\in X_*(T)_\bbR$ we have $\fg\cdot v+\rho(g)F^0_y V=V$.
\end{lem}
\begin{proof}
    Suppose there exists $g\in G$ and $y\in X_*(T)_\bbR$ such that $\fg\cdot v+\rho(g)F^0_y V\ne V$. Then there exists a nonzero linear form $\varphi\in V^*$ that annihilates the subspace $\fg\cdot v+\rho(g)F^0_y V$. Then we have 
    \[\varphi\in (V/\rho(g)F^0_y V)^*=\rho^*(g)F_y^{>0}(V^*)\]
    and so $\varphi$ is $G$-unstable in $V^*$. On the other hand $\varphi$ annihilates $\fg\cdot v$ and hence $v$ is not $G$-good. \par 
    Conversely suppose $v$ is not $G$-good. Then there exists a nonzero $G$-unstable $\varphi\in V^*$ that annihilates $\fg\cdot v$. Since $\varphi$ is $G$-unstable, there exists $g\in G$ and $y\in X_*(T)_\bbR$ such that $\varphi\in\rho^*(g)F_y^{>0}(V^*)=(V/\rho(g)F^0_yV)^*$. Combined with the fact that $\varphi\ne0$ annihilates $\fg\cdot v$, we get $\rho(g)F^0_yV+\fg\cdot v\subset\ker(\varphi)\ne V$. 
\end{proof}
Now we specialize to the case of the adjoint representations.
\begin{prop}\label{prop:rs-good}
    If $G$ is coroot-unramified (resp. root-smooth), then any nonzero element in $\ft^*$ (resp. $\ft$) is not $G$-unstable.\par 
    If $G$ is both coroot-unramified and root-smooth, then any element in $\fg_{\rs}$ or $\fg^*_{\rs}$ is $G$-good.
\end{prop}
\begin{proof}
    First assume that $G$ is coroot-unramified. Let $\xi\in\ft^*$ be an element that is $G$-unstable in $\fg^*$. Then there exists a one-parameter subgroup $\la\colon\bbG_m\to G$ such that $\lim\limits_{t\to0}\Ad^*(\la(t))\xi=0$ and so $f(\xi)=f(0)$ for all $f\in k[\fg^*]^G$. By the dual Chevalley restriction Theorem \ref{thm:dual-Chevalley-restriction} we see that $\varphi(\xi)=\varphi(0)$ for all $\varphi\in k[\ft^*]^W$. Since $W$ is a finite group, this implies that $\xi=0$. Therefore any nonzero element in $\ft^*$ is not $G$-unstable.\par
    When $G$ is root-smooth, one proves by the same argument that any nonzero element in $\ft$ is not $G$-unstable, using the usual Chevalley restriction theorem \ref{thm:chevalley} instead.\par 
    Now assume that $G$ is both coroot-unramified and root-smooth. To show that any element in $\fg_{\rs}$ is $G$-good, it suffices to show that any element $\ga\in\ft_{\reg}$ is $G$-good. We have $\fg\cdot\ga=[\fg,\ga]=\fn\oplus\fn_-$. If $\xi\in\fg^*$ is a $G$-unstable element such that $[\fg,\ga]\subset\ker\xi$, then we have $\xi\in\ft^*$ and hence $\xi=0$ by what we already proved. Thus the element $\ga$ is $G$-good. Similarly, to show that any element in $\fg^*_{\rs}$ is $G$-good it suffices to verify that any element $\eta\in\ft^*_{\reg}$ is $G$-good. We have $\fg\cdot\eta=\ft^\perp$ and hence any element of $\fg$ that annihilates $\fg\cdot\eta$ lies in $\ft$. Since $\ft$ has no nonzero $G$-unstable element, we conclude that $\eta$ is $G$-good.  
\end{proof}

\section{Review on Bruhat-Tits and Moy-Prasad theory}\label{sec:BT-MP}
In this section we set up the notation for the study of Witt vector affine Springer fibers, and review some basic facts from Bruhat-Tits theory and Moy-Prasad theory.
\subsection{Local fields and extensions}
Fix a prime number $p>0$ and let $k$ be an algebraically closed field of characteristic $p$. For any $k$-algebra $R$, we let $W(R)$ be its ring of $p$-typical Witt vectors. Let $\cO_0=W(k)$ and let $F_0=W(k)[\frac{1}{p}]$ be the fraction field of $\cO_0$. Let $F$ be a finite (totally ramified) extension of $F_0$ and let $\cO$ be the ring of integers in $F$. Then $F$ is a complete discrete valued field of characteristic zero with residue field $k$. Fix a (separable) algebraic closure $\overline{F}$ of $F$.\par
Let $F^t$ be the maximal tamely ramified extension of $F$ in $\overline{F}$. Let $\cO^t$ (resp. $\overline{\cO}$) be the valuation ring of $F^t$ (resp. $\overline{F}$). Let $\Gamma\coloneqq\mathrm{Gal}(\overline{F}/F)$ be the absolute Galois group of $F$ (which coincides with the inertia group since the residue field is algebraically closed by assumption). Let $P\coloneqq\mathrm{Gal}(\overline{F}/F^t)$ be the wild inertia group and $\Gamma^t\coloneqq \Gamma/P=\mathrm{Gal}(F^t/F)$ be the tame inertia group.\par 
Let $\bbN^{(p)}\coloneqq\bbN\setminus p\bbN$ be the multiplicative monoid of positive integers prime to $p$. Fix a compatible system of primitive roots of unity $\{\zeta_m,m\in\bbN^{(p)}\}$ in $F$, where each $\zeta_m$ is a primitive $m$-th root of unity such that $\zeta_{mm'}^m=\zeta_{m'}$ for all $m,m'\in\bbN^{(p)}$.\par 
Fix a uniformizer $\varpi\in F$. For each $m\in\bbN^{(p)}$, fix an element $\varpi_m\in F^t$ such that $\varpi_1=\varpi$ and for each $m,m'\in\bbN^{(p)}$ we have $(\varpi_{mm'})^m=\varpi_{m'}$. Denote $F_m\coloneqq F(\varpi_m)$ and let $\cO_m\subset F_m$ be its valuation ring. Then we have $F^t=\bigcup\limits_{m\in\bbN^{(p)}}F_m$.
Let $\mu_m$ be the group of $m$-th roots of unity in $F$ and let 
\[\Hat{\bbZ}(1)^{(p)}\coloneqq\varprojlim_{m\in\bbN^{(p)}}\mu_m\] 
be the group of compatible systems of roots of unity of order prime to $p$ in $F$. Then we have a canonical isomorphism $\Gamma^t\cong\Hat{\bbZ}(1)^{(p)}$ and the compatible system $\zeta\coloneqq(\zeta_m)_{m\in\bbN^{(p)}}$ is a topological generator of $\Gamma^t$ that acts on $F_m$ by the automorphism $\varpi_m\mapsto\zeta_m\varpi_m$.\par 
We let $\val$ be the valuation on $F$ normalized by requiring $\val(\varpi)=1$, and extended uniquely to $\overline{F}$.
\subsection{Reductive groups over local fields}\label{sec:group-local-field}
Let $G$ be a connected reductive group over $F$. Since $F$ is complete and the residue field $k$ is algebraically closed, $G$ is automatically quasi-split by Steinberg's theorem. We will fix a pinning of $G$ as follows.
Let $F'/F$ be the \emph{minimal} Galois extension such that $G$ becomes split over $F'$. Let $\bbG$ be the split connected reductive group scheme over $\bbZ$ with Lie algebra $\bbg\coloneqq\Lie(\bbG)$, together with an isomorphism $\bbG_{F'}\cong G_{F'}$. Fix a pinning $(\bbT,\bbB,(x_\alpha)_{\alpha\in\Delta})$ of $\bbG$ defined over $\bbZ$, consisting of a split maximal torus $\bbT$, a Borel subgroup $\bbB$ containing $\bbT$ and for each simple root $\alpha\in\Delta$ a $\bbZ$-basis vector $x_\alpha\in\bbg_\alpha(\bbZ)$ in the root space $\bbg_\alpha\subset\bbg$. Let $\bbt$ be the Lie algebra of $\bbT$, viewed either as a $\bbZ$-module or a scheme. Let $\sG\coloneqq\bbG_k$ (resp. $\sT\coloneqq\bbT_k$, $\sB\coloneqq\bbB_k$) be the base change of $\bbG$ (resp. $\bbT,\bbB$) to $k$. In the following we will often simplify notation and still use $\bbG$, $\bbT$ etc. to denote their base change to other rings. We hope this will not cause any confusion since in each case it will be clear from the context what are the base rings.\par
The group $\mathrm{Out}(\bbG)$ of outer automorphisms of $\bbG$ can be identified with the subgroup of the automorphism group scheme $\mathrm{Aut}(\bbG)$ that stabilizes the chosen pinning. The group $G$ is the twisted form of $\bbG$ associated to a homomorphism 
\[\rho_G:\mathrm{Gal}(F'/F)\to\mathrm{Out}(\bbG).\]
We also get a Borel subgroup $B$ and maximal torus $T$ of $G$, both defined over $F$. Let $\ft$ (resp. $\fb$) be the Lie algebra of $T$ (resp. $B$). Let $S\subset T$ be the maximal $F$-split sub-torus and let $N$ be the normalizer of $S$ in $G$. Let $\widetilde{W}\coloneqq N(F)/\cT(\cO)$ be the Iwahori-Weyl group of $G$, where $\cT$ is the connected N\'eron model of $T$ (which is also the unique parahoric group scheme of $T$). 

\subsection{Tamely ramified groups}\label{subsec:special-integral-model}
For the remainder of the paper we assume moreover that the splitting extension for $G$ is $F'=F_e$, a \emph{tamely} ramified extension of $F$ of degree $e\in\bbN^{(p)}$. Let us construct a special integral model $\cG_0$ of $G$ and an Iwahori subgroup $\bfI\subset G(F)$ explicitly.\par 
The choice of the uniformizer $\varpi_e\in\cO_e$ as above determines an isomorphism
\begin{equation}\label{eq:Gal-mu-e-isom}
    \mathrm{Gal}(F_e/F)\cong\mu_e
\end{equation}
that identifies $\zeta\in\mu_e$ with the automorphism $\varpi_e\mapsto\zeta\varpi_e$ of $F_e$. Then we get an action of $\mu_e$ on $\bbG$ via a homomorphism
\begin{equation}\label{eq:rho-G}
    \rho_G\colon\mu_e\to\mathrm{Out}(\bbG).
\end{equation}
Let $\cG_0^\dagger\coloneqq(\mathrm{Res}_{\cO_e/\cO}\bbG)^{\mu_e}$ where $\mu_e$ acts simultaneously on $\cO_e$ and $\bbG$ (through $\rho_G$). Let $\cG_0$ be the fiberwise identity component of $\cG_0^\dagger$. Similarly, we let $\cT$ be the fiberwise identity component of $(\mathrm{Res}_{\cO_e/\cO}\bbT)^{\mu_e}$. By \cite{Chi-Witt}*{Lemma 2.1 and Lemma 2.2}, the group schemes $\cG_0,\cG_0^\dagger,\cT$ are affine and smooth over $\cO$. Moreover, $\cT$ is the connected N\'eron model of $T$. The maximal reductive quotient of the special fiber of $\cG_0$ is $\mathsf{G}_0\coloneqq\sG^{\mu_e,\circ}$, the identity component of the fixed point locus of the $\rho_G(\mu_e)$ action on $\sG$. Let $\mathsf{A}\coloneqq\bbT_k^{\mu_e,\circ}$. Then $\sA$ is the maximal reductive quotient of the special fiber of $\cT$. Moreover, $\sG_0$ is a reductive group over $k$ and $\sA$ is a maximal torus of $\sG_0$. \par
By construction there is a natural reduction map $\mathrm{Red}_0:\cG_0^\dagger(\cO)\to\sG(k)$ whose image is $\sG^{\mu_e}(k)$. It restricts to a surjective homomorphism (that we denote by the same symbol)
\begin{equation}\label{eq:special-group-reduction}
    \mathrm{Red}_0:\cG_0(\cO)\to\sG_0(k).
\end{equation}
Let $\bfI\subset\cG_0(\cO)$ be the inverse image of $\bbB(k)\cap\sG_0(k)$ under $\mathrm{Red}_0$. Then $\bfI$ is an Iwahori subgroup of $G(F)$.\par 
If $G$ is split, then $e=1$ and $\cG_0=\cG_0^\dagger=\bbG$.

\subsection{The Bruhat-Tits buildings}\label{sec:BT-building}
Let $\La\coloneqq X_*(\bbT)$ be the coweight lattice of $\bbT$ and denote 
\[\La^{(p)}\coloneqq\La\otimes_\bbZ\bbZ_{(p)},\quad\La_\bbQ\coloneqq\La\otimes_\bbZ\bbQ,\quad\La_\bbR\coloneqq\La\otimes_\bbZ\bbR.\]
Following \cite{KP} we let $\cB(G)$ denote the reduced Bruhat-Tits building of $G$ and let $\widetilde{\cB}(G)$ denote the extended Bruhat-Tits building of $G$. The integral model $\cG_0$ constructed in \S~\ref{subsec:special-integral-model} determines a special point in $\cB(G)$ and we choose a lift $0\in\widetilde{\cB}(G)$ of it. Then we use the point $0$ to identify the apartment of $\widetilde{\cB}(G)$ corresponding to $T$ with the $\mu_e$-invariant subspace $\La_\bbR^{\mu_e}$. The Iwahori subgroup $\bfI\subset G(F)$ defined above corresponds to an alcove in the reduced building $\cB(G)$ whose closure contains the image of the special point $0$. We have an action of $G(F)$ on $\widetilde{\cB}(G)$ such that for any $\mu\in \La^{\mu_e}$, the element $\varpi^\mu\coloneqq\mu(\varpi)$ acts on the apartment $\La_\bbR^{\mu_e}$ by translation by $-\mu$. We refer to \cite{KP}*{Remark 4.1.5, \S4.3.4} for the convention of this negative sign here. 

\subsection{Parahoric group scheme}\label{sec:parahoric-group-scheme}
For each $x\in\widetilde{\cB}(G)$, let $\cG_x$ be the associated Bruhat-Tits parahoric group scheme. Recall that $\cG_x$ is a smooth affine group scheme over $\cO$ with connected geometric fibers. The generic fiber of $\cG_x$ is isomorphic to $G$ and the set of $\cO$-points $\bfP_x\coloneqq\cG_x(\cO)$ is a parahoric subgroup of $G(F)$. Let $\bfP_{x+}$ be the pro-unipotent radical of $\bfP_x$ and let $\sG_x$ be the maximal reductive quotient of the special fiber of $\cG_x$. Then we have an isomorphism of groups $\bfP_x/\bfP_{x+}\cong\sG_x(k)$. When $x$ is the special point $0$ we recover the special integral model $\cG_0$ constructed in \S~\ref{subsec:special-integral-model}.\par
More generally when $x\in\La^{(p),\mu_e}$ these groups can be described more concretely as follows. Choose $m\in\bbN^{(p)}$ such that $e|m$ and $mx\in\La$, and let $\varpi^x\coloneqq mx(\varpi_m)$. By the compatibility of the uniformizers $\varpi_m$ we see that $\varpi^x$ is independent of the choice of $m$. Then we have
\begin{equation}\label{eq:parahoric-group-description}
    \bfP_x=\cG_x(\cO)\subset\mathrm{Ad}(\varpi^{-x})\bbG(\cO_m)\cap G(F)
\end{equation}
where the occurrence of $\varpi^{-x}$ instead of $\varpi^x$ is due to the convention recalled at the end of \S\ref{sec:BT-building}.\par 
One can describe $\sG_x$ alternatively as follows. We view $x$ as a homomorphism:
\[x\colon \hat{\bbZ}(1)^{(p)}\to\bbT(\Bar{k})\]
and consider the homomorphism $\alpha_x:\Hat{\bbZ}(1)^{(p)}\to\mathrm{Aut}(\bbG)$ defined by 
\[\alpha_x(\zeta)=\mathrm{Ad}(x(\zeta))\circ\rho_G(\zeta^{-1}),\quad \text{for all }\zeta\in\Hat{\bbZ}(1)^{(p)}.\] 
Then $\sG_x$ is the identity component of the fixed point locus of $\alpha_x(\Hat{\bbZ}(1)^{(p)})$ in $\sG$.\par
We define a specialization map $\mathrm{Red}_{x}:\bfP_x\to\sG(k)$ to be the composition of the following morphisms
\begin{equation}\label{eq:red-Px}
    \mathrm{Red}_x\colon\bfP_x=\cG_x(\cO)\xrightarrow{\mathrm{Ad}(\varpi^x)}\bbG(\cO_m)\to\sG(k)
\end{equation}
where the last arrow is the reduction modulo $\varpi_m$. Then $\mathrm{Red}_x$ induces an isomorphism
\begin{equation}\label{eq:Px-red-quotient}
    \bfP_x/\bfP_{x+}\cong\sG_x(k)
\end{equation}
that is independent of the choice of $m\in\bbN^{(p)}$ (but dependent on the choice of the compatible system of uniformizers $\varpi_m$). See for example \cite{Wa06}*{Lemme 2.3.1}.

\subsection{Moy-Prasad filtrations}
For each $x\in\widetilde{\cB}(G)$ and real number $r\ge0$, there is a normal open bounded subgroup $\bfP_{x,r}\subset\bfP_x$ such that $\bfP_{x,0}=\bfP_x$ and for any $r<s$, $\bfP_{x,s}\subset\bfP_{x,r}$ is a normal subgroup. In particular, the union 
\[\bfP_{x,r+}\coloneqq\bigcup_{s>r}\bfP_{x,s}\] 
is a normal subgroup of $\bfP_{x,r}$ and $\bfP_{x,0+}=\bfP_{x+}$ is the pro-unipotent radical of $\bfP_x$. The subgroups $\{\bfP_{x,r}, r\ge0\}$ form a fundamental system of neighborhood of identity in $G(F)$.\par 
The successive quotients $\sG_{x}(r)\coloneqq\bfP_{x,r}/\bfP_{x,r+}$ are (the $k$-points of) algebraic groups defined over the residue field $k$. In particular, $\sG_x(0)=\sG_x(k)$ is the reductive quotient of $\bfP_x$ (viewed as a pro-algebraic group over $k$); while if $r>0$, then $\sG_x(r)$ is a vector group (i.e. the underlying additive group of a $k$-vector space) which can be described by the Lie algebra as follows. We often abuse notation by writing $\sG_x(r)$ also for the group of its $k$-points.\par 
Let $\fg_{x,r}$ (resp. $\fg_{x,r+}$) be the Lie algebra of $\bfP_{x,r}$ (resp. $\bfP_{x,r+}$). Then 
\[\fg_x(r)\coloneqq\fg_{x,r}/\fg_{x,r+}\] 
is a $k$-vector space. When $r=0$, we denote $\fg_x\coloneqq\Lie(\bfP_{x})$ and $\fg_{x+}\coloneqq\Lie(\bfP_{x+})$. Then we have a canonical isomorphism $\fg_x/\fg_{x+}\cong\Lie(\sG_x)$.\par
When $r>0$ there is an isomorphism of vector groups (see \cite{KP}*{Theorem 13.5.1}): 
\begin{equation}\label{eq:MP-isom}
    \xi_{x,r}\colon\sG_x(r)=\bfP_{x,r}/\bfP_{x,r+}\xrightarrow{\sim}\fg_x(r)=\fg_{x,r}/\fg_{x,r+}.
\end{equation}

\begin{defn}\label{def:MP-pair}
    A \emph{tame Moy-Prasad pair} $(x,r)$ consists of an element $x\in\La^{(p),\mu_e}$, viewed as a point in $\widetilde{\cB}(G)$, and a number $r\in\bbZ_{(p)}$.\par
    A \emph{tame Moy-Prasad triple} $(x,r,m)$ consists of a tame Moy-Prasad pair $(x,r)$ together with a positive integer $m\in\bbN^{(p)}$ such that  $mx\in\La$, $mr\in\bbZ$ and $e|m$.
\end{defn}
Let $(x,r)$ be a tame Moy-Prasad pair. We can always find $m\in\bbN^{(p)}$ such that $(x,r,m)$ forms a tame Moy-Prasad triple. For each $r\in\frac{1}{m}\bbZ\cap\bbQ_{\ge0}$, the $\cO$-lattice $\fg_{x,r}\subset\fg(F)$ admits a simple description that we now explain. First we describe the analogous $\cO_m$-lattice $\fg(F_m)_{x,r}\subset\fg(F_m)$. Note that $x\in\frac{1}{m}\La^{\mu_e}$ is a hyperspecial point of the extended Bruhat-Tits building for $G(F_m)$. So we have
\[\fg(F_m)_{x,r}=\varpi_m^{mr}\Ad(\varpi^{-x})(\bbg(\cO_m))\] 
and consequently
\[\fg_{x,r}=\fg(F)\cap\fg(F_m)_{x,r}=\varpi_m^{mr}\Ad(\varpi^{-x})(\bbg(\cO_m))\cap\fg(F).\]
Using this formula, we can extend the definition of $\fg_{x,r}$, $\fg_{x,r+}$ and $\fg_x(r)$ to the case $r\in\bbZ_{(p)}$ is not necessarily non-negative.\par
To describe the group $\bfP_{x,r}$ in a similar way, we consider the $n$-th congruence subgroup of $\bbG(\cO_m)$ defined by
\[\bbG(\cO_m)_n\coloneqq\ker(\bbG(\cO_m)\to\bbG(\cO_m/\varpi_m^{n})),\qquad n\in\bbZ_{>0}.\]
Then we have 
\begin{equation}\label{eq:MP-subgp}
    \bfP_{x,r}=\mathrm{Ad}(\varpi^{-x})\bbG(\cO_m)_{mr}\cap G(F),\qquad r\in\frac{1}{m}\bbZ_{>0}.
\end{equation}

\section{Generalization of the Moy-Prasad filtration}\label{sec:generalization-MP-filtration}
In this section we describe a natural generalization of the Moy-Prasad filtration on Lie algebras to general representation spaces. This will provide the natural context for our main theorem. We keep the notation in \S\ref{subsec:special-integral-model}.
\subsection{The setup}\label{subsec:generalized-MP-setup}
Let $\bbV$ be a free $\cO$-module of finite rank and let 
\[\rho\colon\bbG\rtimes\mu_e\to\mathrm{GL}(\bbV)\] 
be an algebraic representation, viewed as a homomorphism of $\cO$-group schemes. Here the semi-direct product is formed under the homomorphism \eqref{eq:rho-G} and the $\cO$-module $\bbV$ is also viewed as an $\cO$-scheme whose functor of points is $\bbV(R)\coloneqq\bbV\otimes_\cO R$ for any $\cO$-algebra $R$. Let $\cV\coloneqq(\bbV\otimes_\cO\cO_e)^{\mu_e}$ where $\mu_e$ acts diagonally (the action on $\bbV$ comes from the representation $\rho$). Then $\cV$ is a finite free $\cO$-module with $\cV\otimes_\cO F_e=\bbV\otimes_\cO F_e$ and $\rho$ induces an algebraic representation 
\[\rho_\cV\colon\cG_0\to\mathrm{GL}(\cV).\]
The restriction of $\rho$ to $\bbT$ induces a decomposition
\begin{equation}\label{eq:V-first-decomposition}
  \bbV=\bigoplus_{\chi\in\sX_\rho}\bbV_\chi  
\end{equation}
where $\sX_\rho\subset X^*(\bbT)$ is the finite $\bbW\rtimes\mu_e$-stable subset of weights of $\rho$ and $\bbV_\chi$ is the finite free $\cO$-submodule of $\bbV$ on which $\bbT$ acts via the character $\chi$. 
We group the decomposition according to $\mu_e$-orbits on $\sX_\rho$ to get
\begin{equation}\label{eq:V-lattice-decmoposition}
    \bbV=\bigoplus_{[\chi]\in\sX_\rho/\mu_e}\bbV_{[\chi]}
\end{equation}
where each summand $\bbV_{[\chi]}\coloneqq\bigoplus\limits_{\chi'\in[\chi]}\bbV_{\chi'}$ is preserved by the $\mu_e$-action on $\bbV$. Then the decomposition \eqref{eq:V-lattice-decmoposition} induces a decomposition of the twisted form
\[\cV=\bigoplus_{[\chi]\in\sX_\rho/\mu_e}\cV_{[\chi]}\]
where $\cV_{[\chi]}\coloneqq(\bbV_{[\chi]}\otimes_\cO\cO_e)^{\mu_e}$.\par 
In the special case where $\bbV=\bbg$ is the adjoint representation on the Lie algebra, we have $\cV=\Lie(\cG_0)$ and the decompositions above are the root space decompositions.  
\subsection{Periodic grading}
Let $\sV\coloneqq\bbV\otimes_\cO k$ and $\sV_\chi\coloneqq\bbV_\chi\otimes_\cO k$ for each $\chi\in\sX_\rho$. Let 
\[\Bar{\rho}=\rho\otimes_\cO k:\sG\rtimes\mu_e\to\mathrm{GL}(\sV)\] 
be the representation induced by $\rho$. \par
For each $x\in\La^{(p),\mu_e}$, viewed as a homomorphism $\hat{\bbZ}(1)^{(p)}\to\bbT$, we define a map 
$\rho_x\colon\hat{\bbZ}(1)^{(p)}\to\mathrm{GL}(\bbV)$ by the formula 
\begin{equation}\label{eq:rho-x}
    \rho_x(\zeta)\coloneqq\rho(x(\zeta)\bar{\zeta}^{-1}),\quad\text{for all }\zeta\in\hat{\bbZ}(1)^{(p)}
\end{equation}
where $\bar{\zeta}$ denotes the image of $\zeta$ in $\mu_e$. Since $x$ is $\mu_e$-invariant, the map $\rho_x$ is a group homomorphism. For any $m\in\bbN^{(p)}$ such that $mx\in\La$ and $e|m$, the homomorphism $\rho_x$ factors through the finite \'etale group scheme $\mu_m$ over $\cO$. Let $\bar{\rho}_x:\hat{\bbZ}(1)^{(p)}\to\mathrm{GL}(\sV)$ be the special fiber of $\rho_x$. Then we get a decomposition of $\sV$ according to the character group $X^*(\hat\bbZ(1)^{(p)})=\bbZ^{(p)}/\bbZ$:
\[\sV=\bigoplus_{\Bar{r}\in\bbZ_{(p)}/\bbZ}\sV(x,\Bar{r})\]
where
\[\sV(x,\Bar{r})\coloneqq\{u\in\sV\mid\bar{\rho}_x(\zeta)u=\zeta^{\Bar{r}}\cdot u,\forall\zeta\in\hat\bbZ(1)^{(p)}\}.\]
In the case of the adjoint representation, we get the cyclic grading on the Lie algebra
\[\Lie(\sG)=\bigoplus_{\bar{r}\in\bbZ_{(p)}/\bbZ}\Lie(\sG)(x,\bar{r}).\]

\subsection{Moy-Prasad lattices}
For any $x\in\La_\bbR^{\mu_e}$ and $r\in\bbR$ we define the following $\cO$-lattice in $V(F)$:
\[V_{x,r}\coloneqq\left(\bigoplus_{\chi\in\sX_\rho}\varpi_e^{\lceil e(r-\chi(x))\rceil}\bbV_\chi(\cO_e)\right)\cap V(F).\]
We have $V_{x,r'}\subset V_{x,r}$ for all $r'>r$. The union
\[\quad V_{x,r+}\coloneqq\bigcup_{r'>r,r'\in\bbR}V_{x,r'}\] 
is an $\cO$-submodule of $V_{x,r}$ and the quotient 
\[V_x(r)\coloneqq V_{x,r}/V_{x,r+}\]
is a finite dimensional $k$-vector space (possibly $0$).\par
In particular, if $(x,r,m)$ is a tame Moy-Prasad triple, then we have
\begin{equation}\label{eq:tame-V-xr}
    V_{x,r}=\cV(F)\cap\rho(\varpi^{-x})(\varpi^r\bbV(\cO_m))
\end{equation}
where the intersection is taken inside $\bbV(F_m)$. Here we recall that 
\[\varpi^{-x}\coloneqq(mx)(\varpi_m)^{-1}\in\bbT(F_m),\quad\varpi^r\coloneqq\varpi_m^{mr}\in F_m.\]

The lattices $V_{x,r}$ and $V_{x,r+}$ are $\bfP_x=\cG_x(\cO)$-stable and their quotient $V_x(r)$ becomes a representation of $\sG_x\cong\bfP_x/\bfP_{x+}$. \par 
In the case where $\cV=\Lie(\cG_0)$ is the adjoint representation of $\cG_0$ and $r\ge0$, the resulting lattice $\fg_{x,r}\subset\fg(F)$ is the Lie algebra of the Moy-Prasad group $\bfP_{x,r}$.\par 
\begin{defn}\label{def:two-pair}
    For any $x,y\in\La_\bbR^{\mu_e}$ and $r,t\in\bbR$, we define
    \[F_y^tV_x(r)\coloneqq\Image[V_{y,t}\cap V_{x,r}\to V_x(r)=V_{x,r}/V_{x,r+}].\]
    In particular in the Lie algebra case we get
    \[F_y^t\fg_x(r)\coloneqq\Image[\fg_{y,t}\cap\fg_{x,r}\to\fg_x(r)=\fg_{x,r}/\fg_{x,r+}].\]
\end{defn}

\subsection{Reduction maps}
For any tame Moy-Prasad pair $(x,r)$, we define the $\cO$-linear map 
\begin{equation}\label{eq:red}
    \mathrm{Red}_{x,r}^V\colon V_{x,r}\to\sV
\end{equation}
to be the composition of the following two maps:
\[V_{x,r}\xrightarrow{\varpi^{-r}\rho(\varpi^x)}\bbV(\cO_m)\xrightarrow{-\otimes_{\cO_m}k}\sV\]
where $m\in\bbN^{(p)}$ is chosen so that $(x,r,m)$ is a tame Moy-Prasad triple, and the second map is reduction mod $\varpi_m\in\cO_m$. The map $\mathrm{Red}_{x,r}^V$ is independent of the choice of $m$. 
\begin{lem}\label{lem:red-V}
    Let $(x,r)$ be a tame Moy-Prasad pair as above. Then the image of $\mathrm{Red}_{x,r}^V$ is equal to $\sV(x,\Bar{r})$ and the kernel of $\mathrm{Red}_{x,r}^V$ is $V_{x,r+}$. In other words, $\mathrm{Red}_{x,r}^V$ induces an isomorphism 
    \begin{equation}\label{eq:reduction-eq}
        \iota_{x,r}^V\colon V_x(r)=V_{x,r}/V_{x,r+}\xrightarrow{\sim}\sV(x,\Bar{r})
    \end{equation}
    where $\Bar{r}\in\bbZ_{(p)}/\bbZ$ is the class of $r$ mod $\bbZ$.\par 
    Moreover, for any $y\in\La^{\mu_e}_\bbR$ and $t\in\bbR$, the isomorphism $\iota_{x,r}^V$ maps the subspace $F_y^tV_x(r)$ in Definition \ref{def:two-pair} to the subspace 
    \[F^{t-r}_{y-x}\sV(x,\bar{r})\coloneqq (F^{t-r}_{y-x}\sV)\cap\sV(x,\bar{r})\]
    where $F^{t-r}_{y-x}\sV$ is the filtration on the $\sG$-module $\sV$ as in Definition \ref{def:filtration}. 
\end{lem}
\begin{proof}
    Choose $m\in\bbN^{(p)}$ such that $(x,r,m)$ is a tame Moy-Prasad triple. Consider the following $\cO_m$-lattices in $\cV(F_m)=\bbV(F_m)$:
    \[V_{x,r}(\cO_m)\coloneqq\rho(\varpi^{-x})(\varpi^r\bbV(\cO_m))=\varpi_m^{mr}\rho(\varpi_m^{-mx})\bbV(\cO_m).\]
    \[V_{x,r+}(\cO_m)\coloneqq\varpi_m^{mr+1}\rho(\varpi_m^{-mx})\bbV(\cO_m).\]
    These lattices are stable under the semi-linear $\mu_m$-action on $\bbV(F_m)=\bbV(F)\otimes_FF_m$ defined as follows: $\zeta\in\mu_m$ acts by $\rho(\bar{\zeta})\otimes\sigma_\zeta$ where $\bar{\zeta}\in\mu_e$ is the image of $\zeta$ and $\sigma_\zeta\in\Gal(F_m/F)$ is the automorphism sending $\varpi_m$ to $\zeta\varpi_m$. By definition we have $\cV(F)=\bbV(F_m)^{\mu_m}$ and
    \[V_{x,r}=V_{x,r}(\cO_m)^{\mu_m},\quad V_{x,r+}=V_{x,r+}(\cO_m)^{\mu_m}.\]
    The reduction map $\mathrm{Red}_{x,r}^V$ induces a $k$-linear isomorphism 
    \[V_{x,r}(\cO_m)/V_{x,r+}(\cO_m)\cong\sV\]
    which intertwines the $\mu_m$ action on the left hand side with the following $\mu_m$-action on the right hand side:
    \[\zeta\cdot w\coloneqq\zeta^{mr}\bar{\rho}_x(\zeta)^{-1}(w),\quad\forall\zeta\in\mu_m,w\in\sV.\]
    After taking $\mu_m$-invariants, we see that $\mathrm{Red}_{x,r}^V$ induces the isomorphism
    \[V_{x,r}/V_{x,r+}=(V_{x,r}(\cO_m)/V_{x,r+}(\cO_m))^{\mu_m}\cong\sV(x,\Bar{r})\]
    where the first equality follows from the fact that $H^1(\mu_m,V_{x,r+}(\cO_m))=0$ since $m$ is invertible in $\cO$.\par 
    It remains to prove the last statement. For any $v\in V_{x,r}$ we let $\tilde{v}\coloneqq\varpi^{-r}\rho(\varpi^x)v\in\bbV(\cO_m)$ so that $\mathrm{Red}_{x,r}^V(v)$ is equal to the image of $\tilde{v}$ in $\sV$. Under the decomposition \eqref{eq:V-first-decomposition}, we write 
    \[\tilde{v}=\sum_{\chi\in\sX_\rho}\tilde{v}_\chi,\quad\tilde{v}_\chi\in\bbV_\chi(\cO_m).\]
    Then we have 
    \[v=\varpi^r\rho(\varpi^{-x})\tilde{v}=\sum_{\chi\in\sX_\rho}\varpi_m^{mr-m\chi(x)}\tilde{v}_\chi.\]
    By definition $v\in V_{y,t}\cap V_{x,r}$ if and only if 
    \[\varpi_m^{mr-m\chi(x)}\tilde{v}_\chi\in \varpi_m^{\lceil m(t-\chi(y))\rceil}\bbV_\chi(\cO_m),\quad\forall \chi\in\sX_\rho.\] 
    Then the image of $\tilde{v}_\chi$ in $\sV$ is zero whenever $\chi(y-x)<t-r$. This shows that the vector $\mathrm{Red}_{x,r}^V(v)$, which equals the image of $\tilde{v}$ in $\sV(x,\bar{r})$, lies in the subspace $F_{y-x}^{t-r}\sV(x,\bar{r})$. Conversely, for any $\bar{w}\in F_{y-x}^{t-r}\sV(x,\bar{r})$ we can first lift it to an element 
    \[\tilde{w}=\sum_{\chi\in\sX_\rho}\tilde{w}_\chi,\qquad \tilde{w}_\chi\in\bbV_\chi(\cO_m)\]
    such that $\tilde{w}_\chi=0$ whenever $\chi(y-x)<t-r$. Then after averaging $\varpi^r\rho(\varpi^{-x})\tilde{w}$ by the $\mu_m$-action we obtain an element $w\in V_{y,t}\cap V_{x,r}$ such that $\mathrm{Red}_{x,r}^V(w)=\bar{w}$ and this completes the proof.
\end{proof}

In the special case of adjoint representation we obtain an $\cO$-linear map
\begin{equation}\label{eq:red-Lie-algebra}
    \mathrm{Red}_{x,r}^\fg\colon\fg_{x,r}\to\Lie(\sG)=\bbg(k)
\end{equation}
that induces an isomorphism of $k$-vector spaces 
\begin{equation}\label{eq:iota-xr-Lie-algebra}
    \iota_{x,r}^\fg\colon \fg_x(r)\coloneqq\fg_{x,r}/\fg_{x,r+}\cong\Lie(\sG)(x,\bar{r}).
\end{equation}

\begin{lem}\label{lem:xi-compatibility}
    Let $(x,r,m)$ be a tame Moy-Prasad triple with $r>0$. Let $\bar{r}\in\bbZ_{(p)}/\bbZ$ be the class of $r$ and put $n=mr\in\bbZ_{\ge1}$. Then the following diagram is commutative:
    \[\xymatrix{
     \bfP_{x,r}/\bfP_{x,r+}\ar[r]^{\xi_{x,r}}_{\sim}\ar[d]_{\Ad(\varpi^x)} & \fg_{x,r}/\fg_{x,r+}\ar[r]^{\iota_{x,r}^\fg}_{\sim}\ar[d]^{\Ad(\varpi^x)} & \Lie(\sG)(x,\bar{r})\ar@{^{(}->}[d]\\
    \bbG(\cO_m)_n/\bbG(\cO_m)_{n+1}\ar[r]^{\xi_{0,n}}_{\sim} & \varpi_m^n\bbg(\cO_m)/\varpi_m^{n+1}\bbg(\cO_m)\ar[r]^{\qquad\iota_{0,n}^{\bbg}}_{\qquad\sim} & \Lie(\sG)
    }\]
    where the upper horizontal maps are the isomorphisms \eqref{eq:MP-isom} and \eqref{eq:iota-xr-Lie-algebra} and the lower horizontal maps are their analogues for the tame Moy-Prasad pair $(0,n)$ of the group $G(F_m)=\bbG(F_m)$. 
\end{lem}
\begin{proof}
    The commutativity of the left square follows from \cite{KP}*{Theorem 13.5.1 (2)} and the commutativity of the right square follows from the definition of the map $\mathrm{Red}_{x,r}^\fg$.
\end{proof}
In general for any $r,s\in\bbZ_{(p)}$, we have the following commutative diagram:
\begin{equation}
\xymatrix{
\fg_{x,r}\otimes_{\cO}V_{x,s}\ar[r]\ar[d] & V_{x,r+s}\ar[d]\\
\Lie(\sG)(x,\bar{r})\otimes\sV(x,\bar{s})\ar[r] & \sV(x,\bar{r}+\bar{s})
}
\end{equation}
where the horizontal maps are induced by the differential of the representation $\rho$ and the vertical maps are induced by the natural reduction maps defined above.

\section{Generalized Witt-vector affine Springer fibers}\label{sec:GASF}
In this section we state and prove the main theorem on the existence of Hessenberg pavings for a class of generalized Witt-vector affine Springer fibers. Then in the end we specialize to the case of Lie algebras and deduce the consequence on Witt vector affine Springer fibers for tame equi-valued elements. \par 
We keep the notation and assumptions in \S\ref{subsec:generalized-MP-setup}.
\subsection{Statement of the main result}\label{subsec:state-main-result}
Recall that for each $y\in\widetilde{\cB}(G)$, the Witt-vector affine flag variety $\mathrm{Fl}_y\coloneqq LG/L^+\cG_y$ is an ind-projective perfect ind-scheme (cf. \cites{BS-Witt,Zhu-mixed}). For each $v\in V(F), y\in\La_\bbR^{\mu_e}\subset\widetilde{\cB}(G)$ and $t\in\bbR$, the \emph{generalized Witt-vector affine Springer fiber} is the perfect sub-indscheme of $\mathrm{Fl}_y=LG/L^+\cG_y$ defined by
\[\mathrm{Fl}_{y,v,t}\coloneqq\{g\in LG/L^+\cG_y\mid\rho(g)^{-1}v\in V_{y,t}\}.\]
We will study properties of $\mathrm{Fl}_{y,v,t}$ under the following assumption on $v$: 
\begin{itemize}
    \item[(*)]  There exists $x\in\La^{(p),\mu_e}$ and $s\in\bbZ_{(p)}$ with $s\ge t$ such that $v\in V_{x,s}$ and $\bar{v}\coloneqq\mathrm{Red}^V_{x,s}(v)\in\sV$ is $\sG$-good in the sense of Definition \ref{def:good-vector}.
\end{itemize}
By Lemma \ref{lem:red-V} we have $\bar{v}\in\sV(x,\bar{s})$, where $\bar{s}\in\bbZ_{(p)}/\bbZ$ is the class of $s$. 
\begin{lem}\label{lem:v-good}
    Suppose that condition (*) is satisfied. For any $g\in\sG_x(k)$, $z\in\La_\bbR^{\mu_e}$ and $r\in\bbZ_{(p)}\cap\bbQ_{\ge0}$ we have
    \[\Lie(\sG)(x,\bar{r})\cdot\bar{v}+\bar{\rho}(g)F^{t-s-r}_{z}\sV(x,\bar{s}+\bar{r})=\sV(x,\bar{s}+\bar{r})\]
    where $\bar{r}\in\bbZ_{(p)}/\bbZ$ is the class of $r$ and 
    \[\bar{\rho}(g)F^{t-s-r}_{z}\sV(x,\bar{s}+\bar{r})\coloneqq (\bar{\rho}(g)F^{t-s-r}_{z}\sV)\cap\sV(x,\bar{s}+\bar{r})\]
    is defined as in Lemma \ref{lem:red-V}.
\end{lem}
\begin{proof}
    The assumptions imply that $t-s-r\le0$ and so $\bar{\rho}(g)F^{t-s-r}_z\sV\supset \bar{\rho}(g)F_z^0\sV$. Then by Lemma \ref{lem:G-good-vector} we get
    \[\Lie(\sG)\cdot\bar{v}+\bar{\rho}(g)F_z^{t-s-r}\sV=\sV.\]
    Since $z\in\La_\bbR^{\mu_e}$, the space $F_z^{t-s-r}\sV$ is preserved by the action of $\bbT_k\rtimes\mu_e$ and thus we have 
    \[F_z^{t-s-r}\sV=\bigoplus_{a\in\bbZ_{(p)}/\bbZ}F_z^{t-s-r}\sV(x,a).\]
    Since $g\in\sG_x(k)$, the automorphism $\bar{\rho}(g)$ preserves the periodic grading and we get that 
    \[\bar{\rho}(g)F_z^{t-s-r}\sV=\bigoplus_{a\in\bbZ_{(p)}/\bbZ}\bar{\rho}(g)F_z^{t-s-r}\sV(x,a).\]
    Moreover, since $\Lie(\sG)(x,a)\cdot\bar{v}\subset\sV(x,\bar{s}+a)$ for all $a\in\bbZ_{(p)}/\bbZ$, the result follows by taking the $(\bar{s}+\bar{r})$-component of the identity above.  
\end{proof}

\begin{thm}\label{thm:main}
    Under the assumption (*), the generalized affine Springer fiber $\mathrm{Fl}_{y,v,t}$ admits a paving by iterations of perfections of affine space bundles over smooth Hessenberg varieties. In particular, the cohomology of $\mathrm{Fl}_{y,v,t}$ is pure. 
\end{thm}
Recall that a perfect indscheme $X$ admits a paving by a certain family $\Omega$ of perfect schemes if there is an exhaustive filtration $X_0\subset X_1\subset X_2\dotsm$ by closed perfect subschemes such that each successive difference $X_i\backslash X_{i-1}$ is a disjoint union of perfect schemes in the family $\Omega$. The family $\Omega$ relevant to us consists of iterations of perfections of affine space bundles, with base the perfection of a smooth Hessenberg variety. \par 
Next we start to prove Theorem \ref{thm:main}. The proof will be completed in \S\ref{sec:affine-space-bundle}. 

\subsection{Stratification}
We first analyze the intersection of $\mathrm{Fl}_{y,v,t}$ with the $L^+\cG_x$-orbits on $\mathrm{Fl}_y$. To simplify notation, we write $\bfP_x\coloneqq L^+\cG_x$, $\bfP_y\coloneqq L^+\cG_y$ and $\bfP_{x,y}\coloneqq\bfP_x\cap\bfP_y$. Each $\bfP_x$-orbit on $\mathrm{Fl}_y$ contains a point of the form $\dot{w}\bfP_y$ where $\dot{w}$ represents an element $w\in\widetilde{W}$ in the Iwahori-Weyl group. The intersection with the $\bfP_x$-orbit through $\dot{w}\bfP_y$ is 
\begin{align*}
    S_w&=\{g\in\bfP_x\dot{w}\bfP_y/\bfP_y\mid \rho(g)^{-1}v\in V_{y,t}\}\\
    &\cong\{h\in\bfP_x/(\bfP_x\cap \dot{w}\bfP_y\dot{w}^{-1})\mid \rho(h)^{-1}v\in \rho(\dot{w})V_{y,t}\}
\end{align*}
where the isomorphism sends $h\dot{w}\bfP_y$ to $h(\bfP_x\cap\dot{w}\bfP_y\dot{w}^{-1})$. Then we see that $S_w$ is isomorphic to the intersection of the $\bfP_x$ orbit through identity on $\mathrm{Fl}_{w(y)}$ with $\mathrm{Fl}_{w(y),v,t}$. Therefore we are reduced to study the spaces
\begin{equation}\label{eq:S}
    \begin{split}
        S&\coloneqq\{g\in\bfP_x/\bfP_{x,y}|\rho(g)^{-1}v\in V_{y,t}\}\\
    &=\{g\in\bfP_x/\bfP_{x,y}|v\in \rho(g)V_{y,t}\cap V_{x,s}\}.
    \end{split}
\end{equation}

Here the second equality holds since $v\in V_{x,s}$ by assumption.\par 
To determine the structure of $S$, we introduce the following spaces for each $r\in\bbZ_{(p)},r\ge0$:
\[\Tilde{S}_r\coloneqq\{g\in\bfP_x/\bfP_{x,y} | v\in \rho(g)V_{y,t}\cap V_{x,s}+V_{x,s+r}\},\]
\[\Tilde{S}_{r+}\coloneqq\{g\in\bfP_x/\bfP_{x,y} | v\in \rho(g)V_{y,t}\cap V_{x,s}+V_{x,s+r+}\}.\]

\begin{lem}\label{lem:first-order-rho}
    Let $r,s\in\bbZ_{(p)}$, with $r>0$. Let $h\in\bfP_{x,r}$ and $u\in V_{x,s}$. Let $X\in\fg_{x,r}$ be any lift of the element $\xi_{x,r}(h)\in\fg_{x,r}/\fg_{x,r+}$ from \eqref{eq:MP-isom}. Then we have
    \[\rho(h)u-u-d\rho(X) u\in V_{x,s+r+}.\] 
\end{lem}
\begin{proof}
    Choose $m\in e\bbN^{(p)}$ such that $ms,mr\in\bbZ$ and $mx\in\La$. Recall that $\varpi^x\coloneqq (mx)(\varpi_m)$. Write $n=mr\ge1$ and put 
    \[h'=\varpi^xh\varpi^{-x}\in\bbG(\cO_m)_{n},\quad u'=\varpi_m^{-ms}\rho(\varpi^x)u\in\bbV(\cO_m)\]
    where we have used \eqref{eq:MP-subgp} and \eqref{eq:tame-V-xr}.\par 
    The map \eqref{eq:MP-isom} for the tame Moy-Prasad pair $(0,n)$ of the group $G(F_m)=\bbG(F_m)$ is an isomorphism 
    \[\xi_{0,n}\colon \bbG(\cO_m)_n/\bbG(\cO_m)_{n+1}\xrightarrow{\sim}\varpi_m^n\bbg(\cO_m)/\varpi_m^{n+1}\bbg(\cO_m).\]
    Let $X'\coloneqq\Ad(\varpi^x)X$. Then $X'\in\varpi_m^n\bbg(\cO_m)$ is a lifting of $\xi_{0,n}(h')$ by Lemma \ref{lem:xi-compatibility}, and so we have
    \[\rho(h')\equiv 1+d\rho(X')\pmod{\varpi_m^{n+1}}\]
    by the functorial property of the isomorphism $\xi_{0,n}$, see \cite{KP}*{Proposition A.5.19(3)}. We deduce that 
    \[\rho(h')u'\equiv u'+d\rho(X')\cdot u'\pmod{\varpi_m^{n+1}\bbV(\cO_m)}.\]
    Applying the operator $\rho(\varpi^{-x})\varpi_m^{ms}$ to the identity above we get that
    \[\rho(h)u\equiv u+d\rho(\Ad(\varpi^{-x})X')u\pmod{\varpi_m^{m(r+s)+1}\rho(\varpi^{-x})\bbV(\cO_m)}\]
    and so
    \[\rho(h)u-u-d\rho(X) u\in\varpi_m^{m(r+s)+1}\rho(\varpi^{-x})\bbV(\cO_m)\cap V(F)=V_{x,s+r+}.\]
\end{proof}
As a consequence of Lemma \ref{lem:first-order-rho}, we see that when $r>0$, the groups $\bfP_{x,r}$ (resp. $\bfP_{x,r+}$) act on the spaces $\tilde{S}_r$ (resp. $\tilde{S}_{r+}$) by left multiplication and we form the coarse quotients:
\[S_r\coloneqq\bfP_{x,r}\backslash\Tilde{S}_r,\quad S_{r+}\coloneqq\bfP_{x,r+}\backslash\Tilde{S}_{r+}\]
which fit in the following commutative diagram in which the horizontal arrows are closed embeddings and the right square is Cartesian and defines the space $Q_r$:
\begin{equation}\label{eq:affine-fibration}
    \xymatrix{
    S_{r+}\ar[r]\ar[rd]_{\pi_r} & Q_r\ar[r]\ar[d]& \bfP_{x,r+}\backslash\bfP_{x}/\bfP_{x,y}\ar[d]^{\Tilde{\pi}_r}\\
    & S_r\ar[r] & \bfP_{x,r}\backslash\bfP_x/\bfP_{x,y}.}
\end{equation}
Here we recall that $\bfP_{x,r}$ is a normal subgroup of $\bfP_x$ and the coarse double quotient $\bfP_{x,r}\backslash\bfP_x/\bfP_{x,y}$ is identified with the quotient of the finite dimensional group $\bfP_x/\bfP_{x,r}$ by its subgroup $\bfP_{x,y}/(\bfP_{x,y}\cap\bfP_{x,r})$ (which acts by right multiplication). When $r$ is sufficiently large so that $\bfP_{x,r}\subset\bfP_{x,y}$ and $V_{x,s+r}\subset V_{y,t}$, we have 
\[S_r=S_{r+}=\Tilde{S}_r=\tilde{S}_{r+}=S,\quad\quad r\gg0.\]

\subsection{Hessenberg varieties}\label{sec:Hess-var}
When $r>0$ is sufficiently small we have $\bfP_{x,r}=\bfP_{x+}$ and $\tilde{S}_r=\tilde{S}_{0+}$. Therefore $\bfP_{x+}$ acts on $\tilde{S}_{0+}$ by Lemma \ref{lem:first-order-rho} and we define the coarse quotient
\[S_{0+}\coloneqq\bfP_{x+}\backslash\tilde{S}_{0+}.\]
Let $\sP_{x,y}\coloneqq\bfP_{x,y}/(\bfP_{x,y}\cap\bfP_{x+})$. Then $\sP_{x,y}$ is a parabolic subgroup of $\sG_x=\bfP_x/\bfP_{x+}$ and the reduction map \eqref{eq:red-Px} induces an isomorphism
\[\bfP_{x+}\backslash\bfP_{x}/\bfP_{x,y}\cong\sG_x/\sP_{x,y}.\]
Let $S_0\coloneqq\Spec(k)$ and $Q_0\coloneqq\sG_x/\sP_{x,y}$. Then the diagram \eqref{eq:affine-fibration} still makes sense for $r=0$ and we have
\begin{align*}
    S_{0+}&=\{g\in\bfP_{x+}\backslash\bfP_x/\bfP_{x,y}\mid \rho(g)^{-1}v\in V_{y,t}\cap V_{x,s}+V_{x,s+}\}\\
    &\cong\{\bar{g}\sP_{x,y}\in\sG_x/\sP_{x,y}\mid\bar{\rho}(\bar{g})^{-1}\bar{v}\in F_{y-x}^{t-s}\sV(x,\bar{s})\}.
\end{align*}
In particular, this defines a canonical deperfection of $S_{0+}$ which we denote by the same symbol. In other words, the vector $\bar{v}\in\sV(x,\bar{s})$ defines a global section of the vector bundle 
\[\cE\coloneqq\sG_x\times^{\sP_{x,y}}(\sV(x,\bar{s})/F^{t-s}_{y-x}\sV(x,\bar{s}))\]
over the partial flag variety $\sG_x/\sP_{x,y}$ and its intersection with the zero section is $S_{0+}$. By assumption (*) and Lemma \ref{lem:v-good}, for any $g\in\sG_x$ representing a point in $S_{0+}$ the map 
\[\Lie(\sG_x)=\Lie(\sG)(x,0)\to \sV(x,\bar{s})/\bar{\rho}(g)F_{y-x}^{t-s}\sV(x,\bar{s}),\quad Y\mapsto Y\cdot\bar{v}\]
is surjective. As explained in \cite{GKM}*{\S2.5}, this shows that the global section of $\cE$ defined by $\bar{v}$ is transverse to the zero section and hence their intersection $S_{0+}$ is a smooth projective variety over $k$, which is called a \emph{Hessenberg variety} in \emph{loc. cit.}

\subsection{Iterated affine space bundle}\label{sec:affine-space-bundle}
Suppose $r>0$ and we will show that the natural map $\pi_r:S_{r+}\to S_r$ is the perfection of an affine space bundle. First let us describe the fibers of $\tilde{\pi}_r$. Let $\sP_x(r)_y$ be the image of $\bfP_y\cap\bfP_{x,r}$ in $\sG_x(r)=\bfP_{x,r}/\bfP_{x,r+}$. Then under the isomorphism $\xi_{x,r}$ in \eqref{eq:MP-isom} we have 
\[\xi_{x,r}(\sP_x(r)_y)=F_y^0\fg_x(r),\]
see Definition \ref{def:two-pair}. Thus we have
\begin{align*}
    \Tilde{\pi}_r^{-1}(g)&\cong\sG_x(r)/\Ad(\bar{g})\sP_x(r)_y\\
    &\xrightarrow[\sim]{\xi_{x,r}}\fg_x(r)/\Ad(\bar{g})\cdot F_y^0\fg_x(r)\\
    &\xrightarrow[\sim]{\eqref{eq:iota-xr-Lie-algebra}}\Lie(\sG)(x,\bar{r})/\Ad(\bar{g})F^{-r}_{y-x}\Lie(\sG)(x,\bar{r})
\end{align*}
where the first isomorphism is given by $g'\mapsto g'g^{-1}$ .\par
Let $\cE_r$ be the base change to $S_r$ of the following vector bundle  
\[\sG_x\times^{\sP_{x,y}}(\Lie(\sG)(x,\bar{r})/F_{y-x}^{-r}\Lie(\sG)(x,\bar{r}))\]
on $\sG_x/\sP_{x,y}$. Let $\cE_r^{\perf}$ be the perfection of $\cE_r$. Then the description above shows that $Q_r$ is an $\cE_r^\perf$-torsor on $S_r$. \par 
For each $g\in\bfP_x$, let $\bar{g}\in\sG_x(k)$ be the image of $g$ under \eqref{eq:Px-red-quotient}. Due to Lemma \ref{lem:red-V}, the isomorphism $\iota_{x,r+s}^V$ from \eqref{eq:reduction-eq} induces an isomorphism
\begin{align*}
    \frac{\rho(g)V_{y,t}\cap V_{x,s}+V_{x,s+r}}{\rho(g)V_{y,t}\cap V_{x,s}+V_{x,s+r+}}&=\frac{V_{x,s+r}}{\rho(g)V_{y,t}\cap V_{x,s+r}+V_{x,s+r+}}\\
    &\cong\sV(x,\bar{s}+\bar{r})/\bar{\rho}(\bar{g})F_{y-x}^{t-s-r}\sV(x,\bar{s}+\bar{r}).    
\end{align*}
If $g\in\bfP_x$ represents a point in $Q_r$, then we have $v\in\rho(g)V_{y,t}\cap V_{x,s}+V_{x,s+r}$ and we let
\begin{equation}\label{eq:vg}
    \bar{v}_g\in\sV(x,\bar{s}+\bar{r})/\bar{\rho}(\bar{g})F_{y-x}^{t-s-r}\sV(x,\bar{s}+\bar{r})
\end{equation} 
denote the image of $v$ under the isomorphism above. By definition, $g$ represents a point in $S_{r+}$ if and only if $\bar{v}_g=0$.\par
Next we determine the relation between $\bar{v}_{hg}$ and $\bar{v}_g$ for $h\in\bfP_{x,r}$. Let $X\in\fg_{x,r}$ be a lifting of $\xi_{x,r}(h)$ and 
put
\[Y_h\coloneqq \iota_{x,r}^\fg(\xi_{x,r}(h))\in\Lie(\sG)(x,\bar{r})\]
where $\iota_{x,r}^\fg$ is the map in \eqref{eq:iota-xr-Lie-algebra} and $\xi_{x,r}$ is the map in \eqref{eq:MP-isom}. \par
Choose a decomposition
\[v=w_g+v_g,\qquad w_g\in \rho(g)V_{y,t}\cap V_{x,s},\quad v_g\in V_{x,s+r}.\]
Then $v_g$ is a lifting of $\bar{v}_g$. We have $v=\rho(h)w_g+(w_g-\rho(h)w_g+v_g)$ and $\rho(h)w_g\in\rho(hg)V_{y,t}\cap V_{x,s}$. By Lemma \ref{lem:first-order-rho} we have 
\[v_g-\rho(h)v_g\in V_{x,s+2r}\subset V_{x,s+r+}\] 
and 
\[w_g-\rho(h)w_g=v-\rho(h)v-(v_g-\rho(h)v_g)\in v-\rho(h)v+V_{x,s+r+}\subset V_{x,s+r}.\]
Thus $v-\rho(h)v+v_g\in V_{x,s+r}$ is a lifting of $\bar{v}_{hg}$ and by Lemma \ref{lem:first-order-rho} we get
\[\bar{v}_{hg}=[\iota_{x,r+s}^V(-d\rho(X) v+v_g)]=[-Y_h\cdot\bar{v}+\bar{v}_g]\]
where $[w]$ denotes the class of a vector $w\in\sV(x,\bar{s}+\bar{r})$ in the quotient space $\sV(x,\bar{s}+\bar{r})/\bar{\rho}(\bar{g})F_{y-x}^{t-s-r}\sV(x,\bar{s}+\bar{r})$.\par 
In particular, if $h\in\bfP_{x,r+}$ then $Y_h=0$ and $\bar{v}_{hg}=\bar{v}_g$. On the other hand, for any $b\in\bfP_{x,y}$ we have $\rho(gb)V_{y,t}=\rho(g)V_{y,t}$ and hence $\bar{v}_{gb}=\bar{v}_g$. Thus the element $\bar{v}_g$ depends only on the image of $g$ in $Q_r$. Let $\cF_r$ be the base change to $S_r$ of the vector bundle
\[\sG_x\times^{\sP_{x,y}}(\sV(x,\bar{s}+\bar{r})/F_{y-x}^{t-s-r}\sV(x,\bar{s}+\bar{r}))\]
 on $\sG_x/\sP_{x,y}$. The discussion above shows that the assignment $g\mapsto \bar{v}_g$ defines a morphism 
\[\phi_r\colon Q_r\to\cF_r^\perf\]
of perfections of affine space bundles over $S_r$ and we have $S_{r+}=\phi_r^{-1}(0)$, the inverse image of the zero section of $\cF_r^\perf$. Moreover, $\phi_r$ determines a homomorphism of vector bundles over $S_r$:
\[L_r\colon\cE_r\to\cF_r\]
whose restriction to the fibers over $g\in S_r$ is given by
\[L_{r,g}\colon\frac{\Lie(\sG)(x,\bar{r})}{\Ad(\bar{g})F^{-r}_{y-x}\Lie(\sG)(x,\bar{r})}\to\frac{\sV(x,\bar{s}+\bar{r})}{\bar{\rho}(\bar{g})F_{y-x}^{t-s-r}\sV(x,\bar{s}+\bar{r})},\quad Y\mapsto -Y\cdot\bar{v}.\]
This is well-defined since the image of $g$ in $\sG_x/\sP_{x,y}$ lies in the Hessenberg variety $S_{0+}$ and hence $\bar{v}\in\bar{\rho}(\bar{g})F_{y-x}^{t-s}\sV(x,\bar{s})$. Let $L_r^\perf\colon\cE_r^\perf\to\cF_r^\perf$ be the induced homomorphism between perfections. Then we have $\phi_r(q+e)=\phi_r(q)+L_r^\perf(e)$ for local sections $q$ (resp. $e$) of $Q_r$ (resp. $\cE_r^\perf$). Then the condition (*) and Lemma \ref{lem:v-good} imply that $L_r$ is surjective and so $S_{r+}\to S_r$ is a torsor $\ker(L_r)^\perf$. Thus the space $S$ from \eqref{eq:S}, which is equal to $S_r$ for $r$ sufficiently large, is an iteration of the perfections of affine space bundles whose base is the smooth Hessenberg variety $S_{0+}$. This finishes the proof of Theorem \ref{thm:main}.

\subsection{The case of Lie algebras}\label{subsec:Lie-alg-case}
We retain the notation from \S\ref{sec:group-local-field}. Assume that $\sG=\bbG_k$ is both root-smooth and coroot-unramified. In particular, the regular semisimple elements in $\Lie(\sG)$ form a nonempty open subset. \par 
For any character $\chi\in X^*(\bbT)$, we let $d\chi\in\bbt^\vee$ denote its differential, which is viewed as a regular function on $\bbt$. Let $\Phi\subset X^*(\bbT)$ be the set of \emph{all} roots of $\bbT$ in $\bbg$. 
\begin{defn}
    An element $\ga\in\bbt(\overline{F})$ is \emph{equi-valued} if there exists $r\in\bbQ$ such that
    \begin{itemize}
        \item $r=\min\{\val(d\chi (\ga))\mid \chi\in X^*(\bbT)\}$ and
        \item $\mathrm{val}(d\alpha (\ga))=r$ for any root $\alpha\in\Phi$.
    \end{itemize}
    The number $r$ is called the \emph{depth} of $\ga$. More generally, an element $\ga\in\fg(F)$ is \emph{equi-valued (of depth $r$)} if it is $\bbG(\overline{F})$-conjugate to an equi-valued element in $\bbt(\overline{F})$ (of depth $r$).
\end{defn}
Note that the second condition in the definition implies that $d\alpha(\ga)\ne0$ for each root $\alpha$. Hence an equi-valued element in $\fg(F)$ is automatically regular semisimple. Moreover, an equi-valued element is bounded (in the sense of \cite{Chi-Witt}*{Definition 4.6}) if and only if its depth is non-negative.
\begin{defn}
    A regular semisimple element $\ga\in\fg(F)$ is \emph{tame} if its centralizer $G_\ga$ splits over a tamely ramified extension of $F$, or equivalently if there exists $m\in\bbN^{(p)}$ such that $\ga$ is $\bbG(F_m)$-conjugate to an element in $\bbt(F_m)$. 
\end{defn}
Let $\ga\in\fg(F)$ be a tame regular semisimple element. Let $m\in\bbN^{(p)}$ be a positive integer such that $e|m$ and $G_\ga$ splits over $F_m$. Let $g\in G(F_m)$ be an element that conjugates $\bbT(F_m)$ to $G_\ga(F_m)$ and let $y=g\cdot 0\in\widetilde{\cB}(G(F_m))$ where 
$0\in\widetilde{\cB}(G)$ is the special point fixed in \S\ref{subsec:special-integral-model}. Then $y$ lies in the apartment of the split maximal torus $G_{\ga,F_m}$ of $G_{F_m}$. Let $x$ be the barycenter of the $\Gal(F_m/F)$-orbit of $y$. Then we have $x\in\widetilde{\cB}(G_{F_m})^{\Gal(F_m/F)}=\widetilde{\cB}(G)$ by tamely ramified descent for Bruhat-Tits buildings. 
For any $\sigma\in\Gal(F_m/F)$ we have $\sigma(y)-y\in\frac{1}{m}X_*(G_{\ga,F_m})$ and hence
\[x-y=\frac{1}{m}\sum_{\sigma\in\Gal(F_m/F)}(\sigma(y)-y)\in\frac{1}{m^2}X_*(G_{\ga,F_m}).\] 
We can choose an apartment of $\widetilde{\cB}(G)$ containing $x$ and the special point $0$ and there is an element of $G(F)$ that takes this apartment to the standard apartment $\La_\bbR^{\mu_e}$. Thus after replacing $\ga$ by a $G(F)$-conjugate we may assume that $x\in\La_\bbR^{\mu_e}$. 
From the identity above we see that $x\in\La^{(p),\mu_e}$ and after enlarging $m$ if necessary, we may assume that $mx\in\La^{\mu_e}$. We denote $\varpi^x\coloneqq (mx)(\varpi_m)$ as before. Then for any $r\in\frac{1}{m}\bbZ$ the following equality holds: 
\[\Ad(\varpi^x)(\fg(F_m)_{x,r})=\varpi_m^{rm}\bbg(\cO_m).\]
Recall that the reduction map \eqref{eq:red-Lie-algebra} is given by the composition
\[\mathrm{Red}_{x,r}^\fg\colon\fg_{x,r}\subset\fg(F_m)_{x,r}\xrightarrow{\varpi_m^{-rm}\Ad(\varpi^x)}\bbg(\cO_m)\to\bbg(k)\]
Since $G_\ga$ is a tamely ramified maximal $F$-torus of $G$, according to \cite{Adler}*{1.9.1} we get
\[\fg_\ga(F)\cap\fg_{x,r}=\fg_\ga(F)_{x,r}=\{X\in\fg_\ga(F)\mid \val(d\chi(X))\ge r,\forall\chi\in X^*(G_{\ga,\overline{F}})\}.\]
Now let us assume moreover that $\ga$ is equi-valued of depth $r$. After enlargine $m$ we may assume that $mr\in\bbZ$. The equality above implies that $\ga\in\fg_{x,r}$ and we deduce that the element $\varpi_m^{-rm}\Ad(\varpi^x)(\ga)$ is equi-valued of depth $0$. Let $\bar{\ga}$ be its image in $\Lie(\sG)$ under the reduction mod $\varpi_m$ map. Then $\bar{\ga}=\mathrm{Red}_{x,r}^\fg(\ga)$ and it is a regular semisimple element in $\Lie(\sG)(k)$, hence $\sG$-good when $\sG$ is both root-smooth and coroot-unramified by Proposition \ref{prop:rs-good}. Then the condition (*) in \S\ref{subsec:state-main-result} is satisfied and we deduce the following consequence from Theorem \ref{thm:main}:
\begin{cor}\label{cor:ASF}
    Let $G$ be a tamely ramified reductive group over $F$. Suppose that the special fiber $\sG=\bbG_k$ of the split form of $G$ is both root-smooth and coroot-unramified. Let $\ga\in\fg(F)$ be a tame equi-valued element of depth $r\ge0$. For any parahoric subgroup $\bfP\subset G(F)$, the Witt-vector affine Springer fiber
    \[X_{\bfP,\ga}=\{g\in G(F)/\bfP\mid\mathrm{Ad}(g)^{-1}(\ga)\in\Lie(\bfP)\}\]
    admits a paving by iterations of perfections of affine space bundles whose base are smooth Hessenberg varieties. In particular, the cohomology of $X_{\bfP,\ga}$ is pure. 
\end{cor}

\begin{rem}
    Although we work in the mixed-characteristic setting, using the same argument our main result Theorem \ref{thm:main} and Corollary \ref{cor:ASF} are also valid for the equal-characteristic field $F=k((\varpi))$, with the same hypotheses on the group $G$ and the elements $v\in V(F)$ (resp. $\ga\in\fg(F)$). In this case the Witt-vector loop spaces are replaced by the ordinary loop spaces and no perfection is needed. Then we obtain more general results in the equal-characteristic setting than \cite{GKM}. 
\end{rem}

\end{document}